\documentclass[12pt]{amsart}
\usepackage{amssymb}
\newcommand{\bc}{{\bf c}}
\newcommand{\h}{\mathfrak{h}}
\newcommand{\p}{\mathfrak{p}}
\newcommand{\C}{\mathbb{C}}
\newcommand{\GL}{\operatorname{GL}}
\newcommand{\Sym}{\mathfrak{S}}
\newcommand{\Z}{\mathbb{Z}}
\newcommand{\gr}{\operatorname{gr}}
\newcommand{\g}{\mathfrak{g}}
\newcommand{\Irr}{\operatorname{Irr}}
\newcommand{\Part}{\mathcal{P}}
\newcommand{\OCat}{\mathcal{O}}

\newcommand{\KZ}{\mathsf{KZ}}
\newcommand{\cont}{\operatorname{cont}}
\newcommand{\End}{\operatorname{End}}
\newcommand{\Hom}{\operatorname{Hom}}
\newcommand{\Cat}{\mathcal{C}}
\newcommand{\Supp}{\operatorname{Supp}}
\newcommand{\Ext}{\operatorname{Ext}}
\newcommand{\Res}{\operatorname{Res}}
\newcommand{\Ind}{\operatorname{Ind}}
\newcommand{\param}{\mathfrak{p}}
\newcommand{\B}{\mathcal{B}}
\newcommand{\WC}{\mathfrak{WC}}
\newcommand{\wc}{\mathfrak{wc}}
\newcommand{\Q}{\mathbb{Q}}
\newcommand{\slf}{\mathfrak{sl}}
\newcommand{\gl}{\mathfrak{gl}}
\newtheorem{Thm}{Theorem}[section]
\newtheorem{Prop}[Thm]{Proposition}
\newtheorem{Cor}[Thm]{Corollary}
\newtheorem{Lem}[Thm]{Lemma}
\theoremstyle{definition}
\newtheorem{Ex}[Thm]{Example}
\newtheorem{defi}[Thm]{Definition}
\newtheorem{Rem}[Thm]{Remark}
\newtheorem{Conj}[Thm]{Conjecture}
\unitlength=1mm   \numberwithin{equation}{section}
\numberwithin{table}{section} \oddsidemargin=0cm
\title{Rational Cherednik algebras and categorification}
\author{Ivan Losev}
\address{Department
of Mathematics, Northeastern University, Boston MA 02115 USA}
\email{i.loseu@neu.edu}
\begin{document}
\begin{abstract}
In this survey article we review Kac-Moody and Heisenberg algebra actions on the categories $\mathcal{O}$ of the rational Cherednik algebras associated to groups  $G(\ell,1,n)$. Using these actions we  solve basic representation
theoretic problems for these categories such as the classification of finite dimensional irreducible
representations and computation of characters of the irreducibles.
\end{abstract}
\maketitle
\section{Introduction}
In this survey we study categorical actions of Lie algebras on categories $\mathcal{O}$
for rational Cherednik algebras. Rational Cherednik algebras were introduced by
Etingof and Ginzburg  in \cite{EG}. These are associative algebras $H_\bc(W)$ constructed
from  complex reflection groups $W$, where $\bc$ is a parameter. They share many common
features with the universal enveloping algebras of semisimple Lie algebras. In particular,
one can consider the category $\mathcal{O}_\bc(W)$ for $H_\bc(W)$ that is similar in many
aspects to the BGG categories $\mathcal{O}$.

For the purposes of this paper, the most important family of complex reflection
groups is the infinite series $G(\ell,1,n)\cong \Sym_n\ltimes (\Z/\ell\Z)^n$.
This is because the category $\mathcal{O}_\bc:=\bigoplus_{n\geqslant 0}\mathcal{O}_\bc(G(\ell,1,n))$
(for some nice, in a sense, most interesting choices of the parameter $\bc$)
carries categorical actions of the Kac-Moody algebras $\hat{\slf}_e$ (\cite{Shan})
and of the Heisenberg algebra $\mathfrak{Heis}$ (\cite{SV}). These actions
provide a categorification of the classical actions of the aforementioned
algebras on the level $\ell$ Fock space $\mathcal{F}^\ell$. In the other direction,
these categorical actions play a crucial role in the study of the categories
$\mathcal{O}_\bc$. Namely, they are of great help in determining the multiplicity
formulas, \cite{GL,RSVV,VV_proof,Webster}, and in the computation of supports
of simple objects in $\mathcal{O}_\bc$ (and, in particular, in classifying the
finite dimensional irreducible $H_\bc(G(\ell,1,n))$-modules), see
\cite{cryst,SV,Cher_supp}.

This paper consists of four sections. In Section \ref{S_Cher} we deal with
the general rational Cherednik algebras. We recall the definitions
of rational Cherednik algebras and of their categories $\mathcal{O}$, recall
a connection between the categories $\mathcal{O}_\bc(W)$ and the Hecke
algebra $\mathcal{H}_{\bf q}(W)$. Then we recall induction and restriction functors
for categories $\mathcal{O}$ constructed in \cite{BE}. We finish by producing
abelian and derived equivalences between categories $\mathcal{O}_\bc(W),
\mathcal{O}_{\bc'}(W)$ for different parameters $\bc,{\bc'}$.
Here we follow \cite{rouqqsch,rouq_der}.

In Section \ref{S_cyclot} we concentrate on the categories $\mathcal{O}_\bc(G(\ell,1,n))$.
We introduce a Kac-Moody categorical action on $\mathcal{O}_\bc$
and its crystal following \cite{Shan}. We explain results from \cite{cryst} that compute the
crystal. Then we explain an equivalence of $\mathcal{O}_{c}(\Sym_n)$ and the category
of modules over a $q$-Schur algebra $S_\epsilon(n,n)$, \cite{rouqqsch,VV_proof}.
Using this equivalence we construct a categorical Heisenberg action on
$D^b(\mathcal{O}_\bc)$ following \cite{SV}.

In Section \ref{S_supp} we will apply categorical actions and wall-crossing
bijections from Section \ref{SS_der_equi} to study the supports
of simple modules. We will see that the supports are described
by two integers that are computed using crystals associated to Kac-Moody
and Heisenberg actions. We will compute the filtration by support on $K_0$
following \cite{SV}. We will also explain combinatorial rules to compute
supports that were found in \cite{Cher_supp}.

In Section \ref{S_cat_equi} we explain a proof of a conjecture of Varagnolo and
Vasserot, \cite{VV}, on an equivalence of the categories $\mathcal{O}_\bc(G(\ell,1,n))$
and certain truncations of affine parabolic categories of type A.
We follow \cite{VV_proof} -- the conjecture was first proved in \cite{RSVV}
but the approach from \cite{VV_proof} relies on the categorical actions
much more than that of \cite{RSVV}. The category equivalence allows
to compute the multiplicities in  $\mathcal{O}_\bc(G(\ell,1,n))$
proving a conjecture of Rouquier, \cite{rouqqsch}.

{\bf Acknowledgements}. I would like to thank Sasha Kleshchev for stimulating discussions.
This work was partially supported by the NSF under grants DMS-1161584, DMS-1501558.

\section{Rational Cherednik algebras and categories $\mathcal{O}$}\label{S_Cher}
\subsection{Rational Cherednik algebras}
In this section we will introduce Rational Cherednik algebras following the
seminal paper of Etingof and Ginzburg, \cite{EG}.

\subsubsection{Complex reflection groups}
Let $\h$ be a finite dimensional vector space over $\C$ and let $W$ be a finite
subgroup of $\GL(\h)$. By a (complex) reflection in $W$ we mean an element
$s$ such that $\operatorname{rk}(s-1)=1$. We write $S$ for the set of all
reflections in $W$. We say that $W$ is a {\it complex reflection group}
if $W$ is generated by $S$. For example, every Coxeter group is a complex
reflection group. Here is another family of examples that is of importance
for this paper.

\begin{Ex}\label{Ex:cyclot_group}
Let $\ell$ and $n$ be positive integers. We form the group $G(\ell,1,n):=\Sym_n\ltimes \mu_\ell^n$,
where $\Sym_n$ stands for the symmetric group on $n$ letters and $\mu_\ell\subset \C^\times$
is the group of $\ell$th roots of $1$. The group $\Sym_n$ acts on $\mu_\ell^n$
by permuting the factors and we use this action to form the semi-direct product. The group $G(\ell,1,n)$
acts on $\h:=\C^n$ as follows. The group $\Sym_n$ permutes the coordinates. Each factor $\mu_\ell$
acts on its own copy of $\C$ by $\eta.z=\eta z, \eta\in \mu_\ell, z\in \C$.

Let us describe the subset $S$ of complex reflections. For $\eta\in \mu_\ell$,
let $\eta_{(k)}$ denote $\eta$ in the $k$th copy of $\mu_\ell\subset G(\ell,1,n)$.
All these elements are reflections. The other reflections are $(ij)\eta_{(i)}\eta_{(j)}^{-1}$,
where $i<j\in \{1,\ldots,n\}$ and $\eta\in \mu_\ell$. Here we write $(ij)$ for the transposition
in $\Sym_n$ that permutes $i$ and $j$. Clearly, $S$ generates $W$ and so $W$ is a complex
reflection group.

We note that, for $\ell=1$, we get $G(1,1,n)=\Sym_n$, while, for $\ell=2$, we get the Weyl group of
type $B$. All other groups $G(\ell,1,n)$ are not Coxeter groups.
\end{Ex}

The classification of complex reflection groups is known thanks to \cite{ST}. There is one
infinite series, $G(\ell,r,n)$, where $r$ divides $\ell$. The group $G(\ell,r,n)$
is the normal subgroup of $G(\ell,1,n)$ consisting of all elements of the form
$\sigma\prod_{i=1}^n \eta_{j,(i)}$, where $\sigma\in \Sym_n$ and $\prod_{i=1}^n \eta_{j}^r=1$.
For example, $G(2,2,n)$ is the Weyl group of type $D_n$. And then there are several exceptional
groups.

\subsubsection{Definition of RCA}
A Rational Cherednik algebra (RCA) depends on a parameter $\bc$, where
$\bc:S\rightarrow \C$ is a conjugation invariant function. We write
$\param$ for the space of all possible parameters $\bc$. Clearly,
$\param$ is a vector space whose dimension equals $|S/W|$, the number
of $W$-conjugacy classes in $S$.

For $s\in S$, let $\alpha_s\in \h^*$ and $\alpha_s^\vee\in \h$ be eigenvectors
for $s$ with eigenvalues different from $1$ (analogs of roots and coroots
for Weyl groups). We partially normalize them by requiring
$\langle\alpha_s,\alpha_s^\vee\rangle=2$.

Recall that for an algebra $A$ and a finite group $\Gamma$ acting on $A$
by automorphisms we can form the {\it smash-product} algebra $A\#\Gamma$.
As a vector space, this algebra is $A\otimes \C\Gamma$, while the product
is given by $(a_1\otimes \gamma_1)\cdot (a_2\otimes \gamma_2)=a_1\gamma_1(a_2)\otimes
\gamma_1\gamma_2$, where $\gamma_1(a_2)$ denotes the image of $a_2$ under the action
of $\gamma_1$. The definition is given in such a way that an $A\#\Gamma$-module
is the same thing as a $\Gamma$-equivariant $A$-module.

Now we define the RCA $H_{\bc}$ as the quotient of $T(\h\oplus \h^*)\#W$
by the following relations:
\begin{equation}\label{eq:RCA_relns}
[x,x']=[y,y']=0,\quad [y,x]=\langle y,x\rangle-\sum_{s\in S}\bc(s)\langle \alpha_s,x\rangle\langle\alpha_s^\vee,y\rangle s,
\quad x,x'\in \h^*, y,y'\in \h.
\end{equation}

Note that we get $H_{0}=D(\h)\#W$.

When we need to indicate the dependence of $H_{\bc}$ on $W$ or on $\h$ and $W$,
we write $H_{\bc}(W)$ or $H_{\bc}(W,\h)$.

Note that we can define the algebra $H_{R,\bc}$ over a $\C$-algebra $R$. For $\bc$
we take a $W$-invariant map $S\rightarrow R$.

\subsubsection{Examples}
Let us give two concrete examples: for groups $\mu_\ell$ and $\Sym_n$.

\begin{Ex}\label{Ex:RCA_cyclic}
Let $W=\Z/\ell \Z$ and $\dim \h=1$. Set $C:=1-2\sum_{s\in \mu_\ell\setminus \{1\}}\bc(s)s$.
Then $H_\bc=\C\langle x,y\rangle\#W/([y,x]=C)$.
\end{Ex}

\begin{Ex}\label{Ex:RCA_symm}
Let $W=\Sym_n$ and $\h$ be its reflection representation, $\h=\{(x_1,\ldots,x_n)\in \C^n| x_1+\ldots+x_n=0\}$.
We have one class of complex reflections, and so $\bc$ is a single complex number
(to be denoted by $c$).
Then $H_c$ is the quotient of $\C\langle x_1,\ldots,x_n,y_1,\ldots,y_n\rangle\#\Sym_n$ by the following relations:
\begin{align*}
&\sum_{i=1}^n x_i=\sum_{i=1}^n y_i=0,\\
&[x_i,x_j]=[y_i,y_j]=0, \forall i,j\in \{1,\ldots,n\},\\
&[y_i,x_j]=c(ij), i\neq j,\\
&[y_i,x_i]=1-c\sum_{j\neq i}(ij).
\end{align*}
\end{Ex}

\subsubsection{PBW property and triangular decomposition}
The algebra $H_{\bc}$ is filtered with $\deg \h^*=\deg W=0, \deg \h=1$
(the filtration by the ``order of a differential operator''). We have
a natural epimorphism $S(\h\oplus \h^*)\#W\twoheadrightarrow \gr H_{\bc}$.

The following fundamental result (the PBW property for Rational Cherednik algebras)
is due to Etingof and Ginzburg, see \cite[Theorem 1.3]{EG}.

\begin{Thm}\label{Thm:PBW_RCA}
The epimorphism $S(\h\oplus \h^*)\#W\twoheadrightarrow \gr H_{\bc}$ is an isomorphism.
\end{Thm}

The following corollary of Theorem \ref{Thm:PBW_RCA} was observed in \cite[Section 3]{GGOR}.

\begin{Cor}\label{Cor:triang}
The multiplication map $S(\h^*)\otimes \C W\otimes S(\h)\rightarrow H_{\bc}$
is an isomorphism of vector spaces.
\end{Cor}

The isomorphism $S(\h^*)\otimes \C W\otimes S(\h)\xrightarrow{\sim} H_{\bc}$
is often called the {\it triangular decomposition}. The reader should
compare it to the triangular decomposition $U(\mathfrak{n}^-)\otimes U(\mathfrak{t})
\otimes U(\mathfrak{n})\xrightarrow{\sim} U(\g)$ for a complex semisimple
Lie algebra $\g$.

\subsubsection{Spherical subalgebras}
Let $e=|W|^{-1}\sum_{w\in W}w$ be the trivial
idempotent in $\C W\subset H_\bc$. Consider the subspace $e H_\bc e\subset H_\bc$.
It is closed under multiplication, and $e$ is a unit. We call $e H_\bc e$ the {\it spherical
subalgebra} of $H_\bc$. Note that $\gr e H_\bc e=e (S(\h\oplus \h^*)\# W)e\cong
S(\h\oplus \h^*)^W$.

We have a functor $H_\bc\operatorname{-mod}\rightarrow e H_{\bc}e\operatorname{-mod},
M\mapsto eM(=M^W)$. This functor is an equivalence if and only if $H_\bc=H_{\bc}e H_\bc$.
In this case, we say that the parameter $\bc$ is {\it spherical}.

It was shown in \cite[Corollary 3.5]{rouq_der} that there is a non-empty  Zariski open subset of spherical
parameters in $\param$.

\subsection{Categories $\mathcal{O}$}
Here we will define the category $\mathcal{O}_\bc(W)$ for the RCA $H_\bc$ following
\cite{GGOR}.

\subsubsection{Definition of the category and Verma modules}\label{SSS_O_Verma}
By definition, the category $\mathcal{O}_\bc(W)$ consists of all modules
with locally nilpotent action of $\h\subset H_\bc$ that are finitely
generated over $H_\bc$ or, equivalently (under the condition
that $\h$ acts locally nilpotently),
over $S(\h^*)\subset H_\bc$.

The triangular decomposition allows one to define Verma modules over
$H_\bc$ that give examples of objects in $\OCat_\bc$. The subalgebra
$S(\h)\# W\subset H_\bc$ plays the role of $U(\mathfrak{b})\subset U(\g)$
and $\C W\subset S(\h)\# W$ plays the role of $U(\mathfrak{t})$.
So the Verma modules are parameterized by the irreducible representations
of $W$: given $\tau\in \Irr(W)$, we set $\Delta_\bc(\tau):=H_\bc\otimes_{S(\h)\#W}\tau$,
where $\h$ acts on $\tau$ by $0$. Thanks to the triangular decomposition,
the Verma module $\Delta_\bc(\tau)$ is naturally isomorphic to $S(\h^*)\otimes\tau$
as a $S(\h^*)\#W$-module (where $W$ acts diagonally and $S(\h^*)$ acts by
multiplications on the left).

We identify $K_0(\OCat_\bc(W))$ with $K_0(W\operatorname{-mod})$ by sending
the class $[\Delta_\bc(\tau)]$ to $[\tau]$.

Let us now give a combinatorial description of the irreducible representations
of $G(\ell,1,n)$.

\begin{Ex}\label{Ex:cyclot_irrep}
The set $\Irr(G(\ell,1,n))$ is in a natural bijection with the set $\Part_\ell(n)$
of $\ell$-multipartitions $\lambda$ of $n$, i.e., the set of $\ell$ partitions $\lambda^{(0)},\ldots,\lambda^{(\ell-1)}$
with $\sum_{i=0}^{\ell-1} |\lambda^{(i)}|=n$. An irreducible module $V_\lambda$
corresponding to $\lambda\in \Part_\ell(n)$ is constructed as follows.

The product $G(\ell,1,\lambda):=\prod_{i=0}^{\ell-1} G(\ell,1,|\lambda^{(i)}|)$ naturally embeds into $G(\ell,1,n)$. Let $V_{\lambda^{(i)}}$ denote the irreducible
$\Sym_{|\lambda^{(i)}|}$-module labeled by the partition $\lambda^{(i)}$. We equip
$V_{\lambda^{(i)}}$ with the structure of a $G(\ell,1,|\lambda^{(i)}|)$-module by making
all $\eta_{(j)}$ act by $\eta^i$. Denote the resulting $G(\ell,1,|\lambda^{(i)}|)$-module
by $V^{(i)}_{\lambda^{(i)}}$. Let $V_\lambda$ denote the $G(\ell,1,n)$-module
induced from the $G(\ell,1,\lambda)$-module $V^{(0)}_{\lambda^{(0)}}\boxtimes
V^{(1)}_{\lambda^{(1)}}\boxtimes\ldots\boxtimes V^{(\ell)}_{\lambda^{(\ell-1)}}$.
The modules $V_{\lambda}$ form a complete collection of the irreducible
$G(\ell,1,n)$-modules.
\end{Ex}

\subsubsection{Euler element and $c$-function}\label{SSS_Euler}
There is a so called {\it Euler element} $h\in H_\bc$ satisfying $[h,x]=x, [h,y]=-y, [h,w]=0$. It is constructed as follows.
Pick a basis $y_1,\ldots,y_n\in \h$ and let $x_1,\ldots,x_n\in \h^*$ be the dual basis. For $s\in S$, let
$\lambda_s$ denote the eigenvalue of $s$ in $\h^*$ different from $1$. Then
\begin{equation}\label{eq:Euler}
h=\sum_{i=1}^n x_i y_i+\frac{n}{2}-\sum_{s\in S} \frac{2\bc(s)}{1-\lambda_s}s.\end{equation}

Every finite dimensional module $M$ lies in $\OCat_\bc(W)$. Indeed, the number of $h$-eigenvalues
in $V$ is finite and so $\h$ acts nilpotently on $V$.

The Euler element acts on $\tau\subset \Delta_\bc(\tau)$ by a scalar denoted by
$c_\tau$ (and called the $c$-function of $\tau$).

We will need to compute the function $c_\tau$ for the groups $G(\ell,1,n)$
(up to a summand independent of $\tau$).  For this, it is convenient to introduce new parameters.

Let $H$ denote a hyperplane of the form $\h^s, s\in S$.
We can find elements $h_{H,j}\in \C$  with $j=0,\ldots,\ell_{H}-1$
and $h_{H,j}=h_{H',j}$ for $H'\in WH$  such that
\begin{equation}\label{eq:c_to_h}c(s)=\sum_{j=1}^{\ell-1}\frac{1-\lambda_s^j}{2}(h_{\h^s,j}-h_{\h^s,j-1}).\end{equation}
Clearly, for fixed $H$, the numbers $h_{H,0},\ldots, h_{H,\ell_H-1}$ are defined up to a common summand.

We can recover the elements $h_{H,i}$ by the formula
\begin{equation}\label{eq:h_to_c} h_{H,i}=\frac{1}{\ell_H}\sum_{s\in W_H\setminus \{1\}}\frac{2c(s)}{\lambda_s-1}\lambda_s^{-i}.
\end{equation}
Note that $\sum_{i=0}^{\ell_H-1}h_{H,i}=0$ in this case.

Note that, for $W=G(\ell,1,n)$ with $\ell,n>1$, we have two conjugacy classes of hyperplanes,
their representatives are given by the equations $x_1=0$ and $x_1=x_2$.
Set $\kappa:=-c(s)$, where $s$ is a reflection in $\Sym_n$, and $h_i:=h_{H,i}$,
where $H$ is the hyperplane $x_1=0$.

Let us get back to computing the functions $c_\lambda$, where $\lambda\in \Part_\ell(n)\cong \Irr(G(\ell,1,n))$.
We view elements of $\Part_\ell(n)$ as $\ell$-tuples of Young diagrams. Let $b$ be a box of $\lambda$.
It can be characterized by three numbers $x,y,i$, where $x$ is the number of column, $y$
is the number of row, and $i$ is the number of the diagram $\lambda^{(i)}$ containing $b$.
Further, we set $c_b:=\kappa\ell (x-y)+\ell h_i$ and $c_\lambda:=\sum_{b\in \lambda}c_b$.
Up to a summand independent of $\lambda$, the number $c_\lambda$ coincides with
$c_{V^\lambda}$ defined above. This is proved in \cite[Section 6.1]{rouqqsch}.

\subsubsection{Simple objects and finiteness properties}
Using the Euler element one can establish many basic structural results
about the categories $\mathcal{O}_\bc(W)$. For example, let us describe the
simple objects in $\mathcal{O}_\bc(W)$.

\begin{Prop}\label{Prop:simples}
Every Verma module $\Delta_\bc(\tau)$ has a unique simple quotient,
to be denoted by $L_\bc(\tau)$. The map $\tau\mapsto L_\bc(\tau)$ is a
bijection between $\Irr(W)$ and $\Irr(\mathcal{O}_\bc(W))$.
\end{Prop}

Also we can establish the following result about the structure
of $\OCat_\bc(W)$. The proof is in \cite[Section 2]{GGOR}, see,
in particular, \cite[Corollaries 2.8,2.16]{GGOR}.

\begin{Prop}\label{Prop:cat_O_prop}
The category $\mathcal{O}_\bc(W)$ has enough projectives and all objects there
have finite length.
\end{Prop}

The two propositions imply that $\mathcal{O}_\bc(W)$ is equivalent to the category of
modules over the finite dimensional algebra $\End_{\OCat_\bc(W)}(P)^{opp}$, where
$P:=\bigoplus_{\tau\in \Irr(W)}P_\bc(\tau)$, and $P_\bc(\tau)$ denotes the projective
cover of $L_\bc(\tau)$.

\subsubsection{Highest weight structure}\label{SSS_HW}
The classical BGG category $\mathcal{O}$ has certain upper triangularity properties
that are formalized in the notion of a highest weight category. Here we will see that the categories $\mathcal{O}_\bc(W)$ are highest weight as well, this result was established in \cite{GGOR}.

Let us start by recalling the general notion of a highest weight category (over a field).
Let $\Cat$ be a $\C$-linear abelian category equivalent to $A\operatorname{-mod}$
for some finite dimensional algebra $A$. For $L\in \Irr(\Cat)$, let $P_L$ denote
the projective cover of $L$. Equip $\Irr(\Cat)$ with a partial order $\leqslant$.
For $L\in \Irr(\Cat)$, consider the subcategory $\Cat_{\leqslant L}$,
the Serre span of $L'\in \Irr(\Cat)$ with $L'\leqslant L$. Let $\Delta_L$
denote the projective cover of $L$ in $\Cat_{\leqslant L}$. The object
$\Delta_L$ coincides with the maximal
quotient of $P_L$ lying in $\Cat_{\leqslant L}$.

\begin{defi}\label{defi_HW}
We say that $\Cat$ is a highest weight category (with respect to the order $\leqslant$)
if, for every $L$,   the kernel of $P_L\twoheadrightarrow \Delta_L$ is filtered by $\Delta_{L'}$ with $L'>L$.
\end{defi}

The objects $\Delta_L$ are called {\it standard}.

Let us get back to the categories $\mathcal{O}_\bc(W)$.
Now fix a parameter $\bc$. Define a partial order $\leqslant_\bc$ on $\Irr(W)$ as follows.
Recall that to $\tau\in \Irr(W)$ we can assign the complex number $c_\tau$ depending on
$\bc$. We set $\tau\leqslant_\bc \xi$ if $\tau=\xi$ or $c_\tau-c_{\xi}\in \Z_{>0}$.

The following result is established in \cite[Theorem 2.19]{GGOR}.

\begin{Prop}\label{Prop:HW}
The category $\OCat_\bc(W)$ is highest weight with respect to the order $\leqslant_\bc$.
The standard $\Delta_{L(\tau)}$ coincides with the Verma module $\Delta_\bc(\tau)$.
\end{Prop}

\subsubsection{Costandard and tilting objects}\label{SSS_cost_tilt}
Recall that in any highest weight category $\Cat$ one has costandard objects $\nabla_L, L\in \Irr(\Cat),$
with $\dim \Ext^i(\Delta_{L},\nabla_{L'})=\delta_{i,0}\delta_{L,L'}$.

In the case of the category $\OCat_\bc(W)$ one can construct the costandard objects $\nabla_\bc(\tau)$ as follows. Consider the parameter $\bc^*$ defined by $ c^*(s):=-c(s^{-1})$. There is an isomorphism $H_\bc(W,\h)\xrightarrow{\sim} H_{\bc^*}(W,\h^*)^{opp}$ that is the identity on $\h^*,\h$  and is the inversion on $W$.
This isomorphism allows to define a contravariant equivalence (the {\it naive duality})
$\OCat_{\bc^*}(W,\h^*)\xrightarrow{\sim}\OCat_{\bc}(W,\h)$ that maps $M$
to its restricted dual $M^\vee:=\bigoplus_{a\in \C} M_a^*$, where we write
$M_a$ for the generalized eigenspace with eigenvalue $a$ of $h$ in $M$.
We set $\nabla_\bc(\tau):=\Delta_{\bc^*}(\tau^*)^\vee$, where $\tau^*$
is the dual to $\tau$.


Recall that by a tilting object in a highest weight category one means an object
that is both standardly filtered (i.e., admits a filtration with standard quotients)
and is costandardly filtered. The indecomposable tilting objects are in bijection
with $\Irr(\Cat)$: for any $L\in \Cat$, there is a unique indecomposable
tilting object $T_L$ that admits an inclusion $\Delta_L\hookrightarrow T_L$
with standardly filtered cokernel. We write $T_\bc(\tau)$ instead of $T_{L_\bc(\tau)}$.

\subsubsection{Generic semisimplicity}\label{SSS_gen_ss}
The $c$-function and the highest weight structure give a sufficient (but not necessary)
criterium for $\OCat_\bc(W)$ to be semisimple.  Let us write $\varpi_{\xi\tau}$
for the element of $\param^*$ given by $\bc\mapsto c_\xi-c_\tau$.

\begin{Lem}\label{Lem:O_ss}
Suppose that  $\varpi_{\xi\tau}(\bc)\not\in \Z\setminus \{0\}$ for all $\xi,\tau\in \Irr(W)$. Then the category $\OCat_\bc(W)$ is semisimple.
\end{Lem}

\subsubsection{Supports}\label{SSS_Supports}
Every object  $M\in\OCat_\bc(W)$ is finitely generated over $S(\h^*)=\C[\h]$.  Thanks to this we can
define the support $\Supp(M)$ of $M$ in $\h$, this will be the support of $M$ viewed as  a
coherent sheaf on $\h$. By definition, this is a closed subvariety in $\h$.

It turns out that $\Supp(M)$ is the union of strata for the {\it stabilizer stratification} of $\h$.
The strata are numbered by the conjugacy
classes of possible stabilizers for the $W$-action on $\h$ (these stabilizers are called
{\it parabolic subgroups} of $W$). Namely, to a parabolic subgroup $\underline{W}\subset W$
we assign the locally closed subvariety $X(\underline{W}):=\{b\in \h| W_b=\underline{W}\}$.
Note that $\overline{X(\underline{W})}:=\bigsqcup_{W'}X(W')$, where the union is taken over
the conjugacy classes of all parabolic subgroups $W'$ containing a
conjugate of $\underline{W}$.

Clearly, for an exact sequence $0\rightarrow M'\rightarrow M\rightarrow M''\rightarrow 0$,
we have $\Supp(M)=\Supp(M')\cup \Supp(M'')$. This, in principle, reduces
the computation of supports to the case of simple modules.  The following result was
proved in \cite[Section 3.8]{BE}.

\begin{Lem}\label{Lem:supp_simple}
Let $L\in \Irr(\OCat_\bc(W))$. Then there is a parabolic subgroup $\underline{W}\subset W$
such that $\Supp(M)=\overline{X(\underline{W})}$.
\end{Lem}

It turns out that $\dim \Supp(L)$ can be recovered from the highest weight structure
on $\mathcal{O}_\bc(W)$. Let $T$ denote the sum of all indecomposable tilting objects in
$\mathcal{O}_\bc(W)$.

\begin{Lem}\label{Lem:supp_tilt}
The number $\dim \h-\dim \Supp(L)$ coincides with the minimal number $i$
such that $\Ext^i_{\mathcal{O}_\bc(W)}(L,T)\neq 0$. It also coincides
with the minimal number $j$ such that $\Ext^i_{\mathcal{O}_\bc(W)}(T,L)\neq 0$.
\end{Lem}
This is \cite[Lemma 6.2]{RSVV}.

\subsubsection{Example}
As an example, consider the category $\mathcal{O}_\bc(W)$ for the group $W=\mu_\ell$.
First of all, note that all Verma modules are isomorphic to $\C[x]$ as $S(\h^*)=\C[x]$-modules.

For $j=0,1,\ldots,\ell-1$, let $\Delta_j$ be the Verma module defined by the representation
$\tau:\eta\mapsto \eta^j$. We are going to compute $y\cdot x^n$ for $x^n\in \Delta_j$. Let
$\eta\in \mu_\ell$ is a generator of $\mu_\ell$. We get
\begin{align*}
y\cdot x^n=&[y,x]\cdot x^{n-1}+x(y\cdot x^{n-1})=(1-2\sum_{i=1}^{\ell-1} c(\eta^i)\eta^i)\cdot x^{n-1}+
x(y\cdot x^{n-1})=\\
&(1-2\sum_{i=1}^{\ell-1}c(\eta^i)\eta^{i(1-n)}\eta^{ij})x^{n-1}+x(y\cdot x^{n-1})=\\
&(n-2\sum_{i=1}^{\ell-1}c(\eta^i)(\eta^{i(1-n)}+\eta^{i(2-n)}+\ldots+1)\eta^{ij})x^{n-1}=\\
&(n+\sum_{i=1}^{\ell-1}\frac{2 c(\eta^i)}{1-\eta^{-i}}\eta^{i(j-n)}-
\sum_{i=1}^{\ell-1}\frac{2c(\eta^i)}{1-\eta^{-i}}\eta^{ij})x^{n-1}=\\
&(n+\ell h_{j-n}-\ell h_j)x^{n-1}.
\end{align*}
Here we write $h_j$ for $h_{H,j}$ and we assume that $h_j$ is numbered mod $\ell$.
In particular, we see that $\Hom_{\mathcal{O}_\bc(W)}(\Delta_k,\Delta_j)$
is one dimensional provided there is $n\geqslant 0$ congruent to
$j-k$ modulo $\ell$ such that $h_j=h_k+n/\ell$, and is zero else.
In particular, either $\Delta_j=L_j$ or $L_j$ is finite dimensional.

\subsection{KZ functor}
Here we will recall the KZ functor introduced in \cite{GGOR}. It connects
the category $\mathcal{O}_\bc(W)$ to the category of modules over the Hecke
algebra of $W$ and is a crucial tool to study $\mathcal{O}_\bc(W)$.

\subsubsection{Localization lemma}
Let $\h^{reg}$ denote the open subset of $\h$ consisting of all $v$ with $W_v=\{1\}$,
equivalently $\h^{reg}=\h\setminus \bigcup_{s\in S}\ker \alpha_s$.
Consider an element $\delta\in \C[\h]^W$ whose set of zeroes in $\h$ coincides
with $\h\setminus \h^{reg}$. We can take $\delta=\left(\prod_{s\in S}\alpha_s\right)^k$,
where $k$ is a suitable integer so that $\delta\in \C[\h]^W$.

Note that $[\delta,x]=[\delta,w]=0$ for all $x\in \h^*, w\in W$. Also note that $[\delta,y]\in
S(\h^*)\#W$ and hence $[\delta,[\delta,y]]=0$. It follows that the endomorphism
$[\delta,\cdot]$ of $H_{\bc}$ is locally nilpotent. So the set $\{\delta^k, k\geqslant 0\}$
satisfies the Ore conditions and we have the localization $H_\bc[\delta^{-1}]$ consisting of
right fractions.

We have an algebra homomorphism $H_\bc\rightarrow D(\h^{reg})\#W$ (the Dunkl homomorphism)
defined on generators $x\in \h^*, w\in W, y\in \h$ as follows:
$$x\mapsto x, \, w\mapsto w, \, y\mapsto y+\sum_{s\in S}\frac{2c(s)\langle\alpha_s,y\rangle}{(1-\lambda_s)\alpha_s}(s-1).$$
This homomorphism factors through $H_\bc[\delta^{-1}]\rightarrow D(\h^{reg})\#W$
because $\delta$ is invertible in $D(\h^{reg})\#W$. The following lemma is easy.

\begin{Lem}\label{Lem:localization}
The homomorphism $H_\bc[\delta^{-1}]\rightarrow D(\h^{reg})\#W$ is an isomorphism.
\end{Lem}

\subsubsection{Localization functor}
Let $M\in \OCat_\bc(W)$. Thanks to Lemma \ref{Lem:localization}, we can view $M[\delta^{-1}]$
as a module over $D(\h^{reg})\#W$. This module is finitely generated over $\C[\h^{reg}]\#W$.
So $M[\delta^{-1}]$ is a $W$-equivariant local system over $\h^{reg}$. Computing $\Delta_\bc(\tau)[\delta^{-1}]$
explicitly, we see that the latter has regular singuularities, see
\cite[Proposition 5.7]{GGOR}. Therefore the same is true for
any $M$. So we get an exact functor $M\mapsto M[\delta^{-1}]$ from $\OCat_\bc(W)$ to the category
$\operatorname{Loc}_{rs}^W(\h^{reg})$ of $W$-equivariant regular local systems on $\h^{reg}$.

Pick a point $p\in \h^{reg}/W$ and let $\pi$ denote the quotient morphism
$\h^{reg}\rightarrow \h^{reg}/W$.  According to Deligne, the category $\operatorname{Loc}_{rs}^W(\h^{reg})$
is equivalent to $\pi_1(\h^{reg}/W,p)\operatorname{-mod}_{fin}$ (the category of finite dimensional
modules) via $N\mapsto [\pi_*(N)^W]_p$.

The group $\pi_1(\h^{reg}/W,p)$ is known as the braid group of $W$
and is denoted by $B_W$. When $W$ is  a Coxeter group, we get the classical
braid group. In general, the structure of  $B_W$ was studied in \cite{BMR}.
The group $B_W$ is generated by elements $T_H$, where $H$ runs over the
set of reflection hyperplanes for $W$. Roughly speaking, $T_H$ is a curve
obtained by rotating a point in $\h^{reg}$ close to $H$ by the angle
of $2\pi/|W_H|$ (here $W_H$ is the pointwise stabilizer of $H$).
Note that if $H,H'$ are $W$-conjugate, then $T_H,T_{H'}$ are conjugate in
$B_W$.

\begin{Ex}\label{Ex:braid_cyclot}
Let us consider the case when $W=G(\ell,1,n)$. When $\ell=1$, we get the usual
type A braid group. When $\ell>1$, we get the affine braid group of type A.
It is given by generators $T_0,\ldots,T_{n-1}$ subject to the following
relations: $T_iT_j=T_jT_i$ when $|i-j|>1$, $T_iT_{i+1}T_i=T_{i+1}T_iT_{i+1}$
for $i>1$, and $T_0T_1T_0T_1=T_1T_0T_1T_0$.
\end{Ex}

\subsubsection{Hecke algebras}
We want to determine the essential image of the functor
$\OCat_\bc(W)\rightarrow B_W\operatorname{-mod}_{fin}$. It turns out
that this image coincides with $\mathcal{H}_{\bf q}(W)\operatorname{-mod}_{fin}
\hookrightarrow B_W\operatorname{-mod}_{fin}$, where $\mathcal{H}_{\bf q}(W)$
is a quotient of $\C B_W$ called the Hecke algebra of $W$, and ${\bf q}$
is a parameter recovered from $\bc$.

First, let us explain how the parameter ${\bf q}$ is computed. It is a collection
$q_{H,i}$ of nonzero complex numbers, where $H$ is a reflection hyperplane
for $W$, $i=0,1,\ldots,|W_H|-1,$ and $q_{H,i}=q_{H',i}$ if $H$ and $H'$
are $W$-conjugate. We set
\begin{equation}\label{eq:h_to_q}q_{H,j}:=\exp(2\pi\sqrt{-1}(h_{H,j}+j/\ell_H)),\end{equation}
where $h_{H,j}$ is recovered from (\ref{eq:c_to_h}). Note that the parameters
$q_{H,0},\ldots, q_{H,|W_H|-1}$ are defined up to  a common multiple.

Following \cite[4C]{BMR}, define the Hecke algebra $\mathcal{H}_{\bf q}(W)$ as the quotient of
$\C B_W$ by the following relations
\begin{equation}\label{eq:Hecke_relns}
\prod_{i=1}^{|W_H|} (T_{H}-q_{H,i})=0,
\end{equation}
where $H$ runs over the set of reflection hyperplanes for $W$.

When $W$ is a Coxeter  group, we get the usual Iwahori-Hecke algebra.

\begin{Ex}\label{Ex:cyclot_Hecke}
Let us consider the case of $W=G(\ell,1,n)$. Then $\mathcal{H}_{\bf q}(W)$ is the cyclotomic Hecke
algebra (a.k.a. Ariki-Koike algebra). This algebra is the quotient of $\C B_W$ by
the following relations: $(T_i+1)(T_i-q)=0, i=1,\ldots,n-1,$ and $\prod_{i=0}^{\ell-1}(T_0-Q_i)$,
where $q=\exp(2\pi\sqrt{-1}\kappa)$ and $Q_i=\exp(2\pi\sqrt{-1}(h_i+i/\ell))$. Recall
that the parameters $\kappa,h_0,\ldots,h_{\ell-1}$ were introduced in
\ref{SSS_Euler}.   Note that $\dim \mathcal{H}_{\bf q}(W)=|W|$
in this case, \cite{AK}.
\end{Ex}

It was shown in \cite[Theorem 5.13]{GGOR} that the functor $\OCat_\bc(W)\rightarrow B_W\operatorname{-mod}_{fin}$
decomposes as the composition of $\operatorname{KZ}_\bc:\OCat_\bc(W)\rightarrow \mathcal{H}_{\bf q}(W)\operatorname{-mod}_{fin}$
(called the KZ functor) and the inclusion of $\mathcal{H}_{\bf q}(W)\operatorname{-mod}_{fin}
\hookrightarrow B_W\operatorname{-mod}_{fin}$.

\begin{Prop}\label{Prop:KZ_surj}
The functor $\operatorname{KZ}_\bc:\OCat_\bc(W)\rightarrow \mathcal{H}_{\bf q}(W)\operatorname{-mod}_{fin}$
is essentially surjective.
\end{Prop}

It is expected (and is shown in many cases) that $\dim \mathcal{H}_{\bf q}(W)=|W|$.
Proposition \ref{Prop:KZ_surj} was proved in \cite[Theorem 5.15]{GGOR} under that assumption.
The assumption was removed in \cite{Hecke_fin_dim}.

\subsubsection{Properties of KZ functor}
Let us list some properties of the KZ functor obtained in \cite[Section 5]{GGOR}.

\begin{Prop}\label{Prop:KZ_properties} The following is true.
\begin{enumerate}
\item The KZ functor $\OCat_\bc(W)\twoheadrightarrow \mathcal{H}_{\bf q}(W)\operatorname{-mod}_{fin}$ is a quotient functor.
Its kernel is the subcategory $\OCat_{\bc,tor}(W)\subset \OCat_\bc(W)$ consisting of all modules in
$\OCat_\bc(W)$ that are torsion over $\C[\h]$ (equivalently, whose support is a proper subvariety in
$\h$).
\item The functor $\operatorname{KZ}_\bc$ is defined by a projective object $P_{KZ}$
in $\mathcal{O}_\bc(W)$ that is also injective. The multiplicity of $\Delta_\bc(\tau)$
in $P_{KZ}$ equals $\dim \tau$.
 \item $\operatorname{KZ}_\bc$ is fully faithful on the projective objects in $\OCat_\bc(W)$.
It is also faithful on the tilting objects in $\OCat_\bc(W)$.
\item Suppose that the parameter ${\bf q}$ satisfies the following condition:
for any reflection hyperplane $H$, we have $q_{H,i}\neq q_{H,j}$
for $i\neq j$. Then $\operatorname{KZ}_\bc$ is fully faithful on all standardly
filtered objects.
\end{enumerate}
\end{Prop}

For $W=G(\ell,1,n)$, the condition in 4) means that $q\neq -1$ and $Q_i\neq Q_j$ for $i\neq j$.


\subsection{Induction and restriction functors}
In this section we review another major tool to study the categories $\OCat_\bc(W)$,
the induction and restriction functors introduced by Bezrukavnikov and Etingof
in \cite{BE}. These functors relate categories $\mathcal{O}_\bc(W)$ and $\mathcal{O}_\bc(\underline{W})$,
where $\underline{W}$ is a parabolic subgroup in $W$. More precisely, we have functors
$\Res_W^{\underline{W}}:\OCat_\bc(W)\rightarrow \OCat_\bc(\underline{W})$ and
$\Ind_{\underline{W}}^W:\OCat_\bc(\underline{W})\rightarrow \OCat_\bc(W)$.
When we write $\OCat_\bc(\underline{W})$, by $\bc$ we mean the restriction
of $\bc$ to $\underline{W}\cap S$. Since we also  consider restriction/induction
functors for other categories, we will sometimes write $\,^{\OCat}\Res_W^{\underline{W}},
\,^{\OCat}\Ind_{\underline{W}}^W$.

The construction of the  functors is technical. We are not going to explain the construction,
we will just explain the properties.

\subsubsection{Exactness and adjointness properties}
We have the following properties of $\Res$ and $\Ind$.

\begin{Prop}\label{Prop:ResInd_biadj}
The functors $\Res^W_{\underline{W}}$ and $\Ind^{\underline{W}}_W$ are biadjoint.
Hence they are exact.
\end{Prop}

The claims that $\Res^W_{\underline{W}}$ and $\Ind^{\underline{W}}_W$ are exact
was checked in \cite[Section 3.5]{BE}. The claim that $\Ind^{\underline{W}}_W$
is right adjoint to $\Res^W_{\underline{W}}$ is a consequence of the construction
in {\it loc.cit.}. The other adjointness was established in \cite[Section 2.4]{Shan}
under some restrictions on $W$ and in \cite{fun_iso}, in general.

In fact, \cite{fun_iso} proved a stronger statement. Recall the parameter $\bc^*$,  contravariant
equivalences $M\mapsto M^\vee$ (taking the restricted duals) $\OCat_\bc(W,\h)\xrightarrow{\sim}
\OCat_{\bc^*}(W,\h^*)$ and $\OCat_\bc(\underline{W},\underline{\h})\xrightarrow{\sim}
\OCat_{\bc^*}(\underline{W},\underline{\h}^*)$ from  \ref{SSS_cost_tilt}. Here $\underline{\h}(=\h/\h^{\underline{W}})$
is the reflection representation of $W$.

\begin{Prop}\label{Prop:ResInd_dual} The functors $\Res^W_{\underline{W}},\Ind_W^{\underline{W}}$ intertwine the duality
functors, i.e., we have functorial isomorphisms $\Res^W_{\underline{W}}(M^\vee)\cong
\Res^W_{\underline{W}}(M)^\vee$ and $\Ind_W^{\underline{W}}(N^\vee)\cong
\Ind_W^{\underline{W}}(N)^\vee$.
\end{Prop}

This is the main result of \cite{fun_iso}.

%

\subsubsection{Relation to KZ}
We have a natural homomorphism $\mathcal{H}_{\bf q}(\underline{W})\rightarrow \mathcal{H}_{\bf q}(W)$.
This gives rise to an exact restriction functor $\,^{\mathcal{H}}\Res^W_{\underline{W}}:
\mathcal{H}_{\bf q}(W)\operatorname{-mod}_{fin}\rightarrow
\mathcal{H}_{\bf q}(\underline{W})\operatorname{-mod}_{fin}$.

The following proposition is \cite[Theorem 2.1]{Shan}.

\begin{Prop}\label{Prop:KZ_Res}
The KZ functors intertwine the restriction functors: $\underline{\operatorname{KZ}}_\bc\circ
\,^{\OCat}\Res^W_{\underline{W}}\cong \,^{\mathcal{H}}\Res^W_{\underline{W}}\circ \operatorname{KZ}_\bc$.
Here we write $\underline{\operatorname{KZ}}_\bc$ for the KZ functor $\OCat_\bc(\underline{W})\rightarrow
\mathcal{H}_{\bf q}(\underline{W})$.
\end{Prop}

As was explained in \cite{Hecke_fin_dim}, the algebra $\mathcal{H}_{\bf q}(W)$ has the maximal
finite dimensional quotient to be denoted by $\mathcal{H}^f_{\bf q}(W)$ here. The same,
of course, applies to $\mathcal{H}_{\bf q}(\underline{W})$. We get induction,
$\,^{\mathcal{H}}\Ind^W_{\underline{W}}$, and coinduction,
$\,^{\mathcal{H}}\operatorname{Coind}^W_{\underline{W}}$, functors
$\mathcal{H}_{\bf q}(\underline{W})\operatorname{-mod}_{fin}\rightarrow
\mathcal{H}_{\bf q}(W)\operatorname{-mod}_{fin}$ associated to the homomorphism
$\mathcal{H}^f_{\bf q}(\underline{W})\rightarrow \mathcal{H}^f_{\bf q}(W)$.
As explained in \cite[Section 2.8]{rouq_der}, 4) has the following corollary.

\begin{Cor}\label{Cor:KZ_Ind}
We have an isomorphism of functors $\,^{\mathcal{H}}\Ind^W_{\underline{W}}\cong \,^{\mathcal{H}}\operatorname{Coind}^W_{\underline{W}}$. The KZ functors intertwine
the induction functors.
\end{Cor}

We have one more  corollary of Proposition \ref{Prop:KZ_Res},
see \cite[Corollary 2.5]{Shan}.

\begin{Cor}\label{Cor:Res_trans} Let $\underline{W}'\subset \underline{W}$ be a parabolic subgroup.
Then we have an isomorphism of functors $\Res^W_{\underline{W}'}\cong
\Res^{\underline{W}}_{\underline{W}'}\circ \Res^W_{\underline{W}}$.
An analogous claim holds for the induction functors.
\end{Cor}

\subsubsection{Behavior on $K_0$ and on (co)standardly filtered objects}
\begin{Lem}\label{Lem:IndRes_stand}
The object $\Res^W_{\underline{W}}(\Delta_\bc(\tau))$ (resp.,
$\Res^W_{\underline{W}}(\nabla_\bc(\tau))$) is  filtered
with quotients $\Delta_\bc(\tau')$ (resp., $\nabla_\bc(\tau')$) for
$\tau'\in \Irr(\underline{W})$. The multiplicity of $\Delta_\bc(\tau')$
(or of $\nabla_\bc(\tau')$) in the filtration equals
$\dim \Hom_{\underline{W}}(\tau',\tau)$. The similar claims are true for
$\Ind_W^{\underline{W}}(\Delta_\bc(\tau'))$ and
$\Ind_W^{\underline{W}}(\nabla_\bc(\tau'))$.
\end{Lem}
\begin{proof}
The claim about $\Res^W_{\underline{W}}(\Delta_\bc(\tau))$ is quite straightforward
from the construction of the functor, see \cite[Proposition 1.9]{Shan}. The claim about
$\Res^W_{\underline{W}}(\nabla_\bc(\tau))$ follows from Proposition \ref{Prop:ResInd_dual}.
The claims about the induction
functors follow from the biadjointness.
\end{proof}

In particular, the map $[\Res^W_{\underline{W}}]:K_0(\OCat_\bc(W))=K_0(W\operatorname{-mod})
\rightarrow K_0(\OCat_\bc(\underline{W}))=K_0(\underline{W}\operatorname{-mod})$
coincides with the usual restriction $K_0(W\operatorname{-mod})
\rightarrow K_0(\underline{W}\operatorname{-mod})$. The similar
claim holds for the induction. This was observed already in \cite[Section 3.6]{BE}.

%
%

\subsubsection{Behavior on supports}
Recall that the support of a module from category $\mathcal{O}_\bc(W)$ was defined in
\ref{SSS_Supports}. Here we will investigate the interplay between the supports
and the induction and restriction functors.

Let us start with the restriction functor. The following result was established
in \cite[Proposition 2.2]{SV}.

\begin{Lem}\label{Lem:Supp_Restr}
Let $M\in \OCat_\bc(W)$ be such that $\Supp(L)=\overline{X(W')}$.
Let $\underline{W}$ be a parabolic subgroup in $W$. Then
$$\Supp(\Res^W_{\underline{W}}(M))=\bigcup_{\underline{W}'}\underline{X}(\underline{W}').$$
Here the union is taken over all ($\underline{W}$-conjugacy classes of) parabolic
subgroups $\underline{W}'\subset \underline{W}$ that are conjugate to $W'$ in $W$.
By $\underline{X}(\underline{W}')$, we denote the stratum of $\underline{\h}$
(the reflection representation of $\underline{W}$) corresponding to the stabilizer $\underline{W}'$.
\end{Lem}

In particular, this lemma implies that $\Res^W_{\underline{W}}(M)=0$ if and only
if $W'$ is not conjugate to a subgroup of $\underline{W}$, while $\Res^W_{\underline{W}}(M)$
is finite dimensional if and only if $W'$ is conjugate to $\underline{W}$.

Let us proceed to the induction functors. The following result is an easy consequence
of  \cite[Proposition 2.7]{SV}.

\begin{Lem}\label{Lem:Supp_Ind}
Let $\underline{L}$ be a simple in $ \OCat_\bc(\underline{W})$ and $\underline{W}'\subset \underline{W}$
be such that $\Supp(\underline{L})=\underline{X}(\underline{W}')$.
Then for every  quotient/sub $M$ of $\Ind_W^{\underline{W}}(\underline{L})$,
we have $\Supp(M)=X(\underline{W}')$.
\end{Lem}

\subsection{Abelian equivalences}\label{SS_abelian_equi}
Here we will establish some equivalences between categories $\mathcal{O}_\bc(W)$
and $\mathcal{O}_{\bc'}(W)$ with different parameters $\bc,\bc'$.
\subsubsection{Twist by a one-dimensional character}
Now let $\chi$ be a one-dimensional character of $W$.
Given $\bc\in \param$, define $\bc^\chi\in \param$ by $ c^\chi(s)=\chi(s)^{-1}c(s)$.
We have an isomorphism $\psi_\chi: H_\bc\xrightarrow{\sim}H_{\bc^\chi}$ given on
the generators by $x\mapsto x, y\mapsto y, w\mapsto \chi(w)w$. This gives rise
to an equivalence $\psi_{\chi*}:\OCat_\bc(W)\xrightarrow{\sim}\OCat_{\bc^\chi}(W)$
that maps $\Delta_\bc(\tau)$ to $\Delta_{\bc^\chi}(\chi\otimes \tau)$.

\subsubsection{Chambers and walls}\label{SSS_chamb_walls}
From now on, in this section we will consider parameters $\bc,\bc'$ giving the same parameter
${\bf q}$ for the Hecke algebra. In fact, there
is the lattice $\param_{\Z}\subset \param$ such that    parameters $\bc,\bc'\in\param$
give the same parameter ${\bf q}$ if and only if $\bc'-\bc\in \param_{\Z}$. We will
see that $\mathcal{O}_\bc(W)$ and $\mathcal{O}_{\bc'}(W)$ are equivalent provided
$\bc,\bc'$ lie in the same ``chamber''.  Let us explain what kind of chambers we consider.

Consider the shifted lattice $\bc+\param_{\Z}$. Set $\Pi_{\xi\tau}:=\ker\varpi_{\xi\tau}$
(recall that the functions $\varpi_{\xi\tau}\in \param^*$ were defined in \ref{SSS_gen_ss}).
Consider the hyperplanes $\Pi_{\xi\tau}$ such that $\varpi_{\xi\tau}(\bc)\in \Q$.
We call them {\it walls} for $\bc+\param_{\Z}$ (note that the walls do not
need to intersect $\bc+\param_{\Z}$). In any case, the hyperplanes $\Pi_{\xi\tau}$
split $\bc+\param_{\Z}$ into the union of chambers that are polyhedral cones.

\subsubsection{Main results}
We have the following.

\begin{Prop}\label{Prop:cat_equi}
Suppose that $\bc'\in \bc+\param_{\Z}$ lies in the closure of the chamber
containing $\bc$. Then there is an abelian equivalence $\Phi_{\bc'\leftarrow \bc}:\OCat_\bc(W)
\xrightarrow{\sim} \OCat_{\bc'}(W)$ intertwining the KZ functors.
\end{Prop}

In fact, under an additional condition one can also ensure that an equivalence
from Proposition \ref{Prop:cat_equi} maps $\Delta_\bc(\tau)$ to $\Delta_{\bc'}(\tau)$.
Namely, according to \cite[Section 7.2]{BC}, there is a group homomorphism
$\param_{\Z}\rightarrow \operatorname{Bij}(\Irr(W))$ called
the KZ twist and denoted by $\mathsf{tw}$. Set $\underline{\param}_{\Z}:=\ker\mathsf{tw}$.
For example, for $W=G(\ell,1,n)$, the KZ twist is trivial, see \cite[6.4.7]{GL}
for explanation and references, while for some other
groups, such as $G_2$, it is not.

\begin{defi}\label{def:hw_equi}
For two highest weight categories $\OCat^1,\OCat^2$ with identified sets of irreducibles
$\Irr(\OCat^1)\cong \Irr(\OCat^2)\cong \Lambda$, by a {\it highest weight} equivalence
$\OCat^1\xrightarrow{\sim}\OCat^2$, we mean an equivalence $\OCat^1\xrightarrow{\sim}\OCat^2$
of abelian categories that maps $\Delta^1_L$ to $\Delta^2_L$, for every $L\in \Lambda$.
\end{defi}

\begin{Prop}\label{Prop:cat_equi_hw}
Under the assumptions of Proposition \ref{Prop:cat_equi}, suppose, in addition, that
$\bc'-\bc\in \underline{\param}_{\Z}$. Then one can take a highest weight
equivalence $\Phi_{\bc'\leftarrow \bc}$ in Proposition \ref{Prop:cat_equi}.
\end{Prop}

Proposition \ref{Prop:cat_equi} is a consequence of Proposition \ref{Prop:cat_equi_hw}.
In fact, for $\psi\in \param_{\Z}$, we have {\it shift functors} $\mathsf{Sh}_\psi:
\OCat_\bc(W)\rightarrow \OCat_{\bc+\psi}(W)$, they are given by taking tensor products
with shift bimodules, see, e.g., \cite[Section 3.1]{rouq_der}. For a fixed $\psi$, this functor is an
equivalence for a Zariski generic $\bc$, this can be deduced from  \cite[Corollary 3.5]{rouq_der}.

\subsubsection{Highest weight covers}\label{SS_hw_covers}
A general technique of proving results like Proposition \ref{Prop:cat_equi_hw} is
due to Rouquier, \cite[Section 4.2]{rouqqsch}, who proved Proposition \ref{Prop:cat_equi_hw}
under the assumption that $\operatorname{KZ}_\bc, \operatorname{KZ}_{\bc'}$ are fully faithful
on all standardly filtered objects, see (4) of Proposition \ref{Prop:KZ_properties}.
An extension of Rouquier's technique found in \cite[Section 7]{VV_proof} allows to remove
the additional restriction on $\bc$.

One obvious obstruction to an existence of an equivalence $\mathcal{O}_\bc(W)\rightarrow
\mathcal{O}_{\bc'}(W)$ with $\Delta_\bc(\lambda)\mapsto \Delta_{\bc'}(\lambda)$
is that highest weight orders  may be different. The condition that ${\bf c}'$ lies in the closure
of the chamber containing ${\bf c}$ guarantees  that $\leqslant^{\bf c}$ refines
$\leqslant^{\bf c'}$ (strictly speaking, here we need to modify the order
and require $\tau<^\bc \xi$ if $c_\tau-c_\xi\in \mathbb{Q}_{>0}$, not in
$\mathbb{Z}_{>0}$, but this does not matter). So there is a common order on
$\Irr(\OCat_\bc(W))\cong \Irr(\OCat_{\bc'}(W))$ making $\OCat_\bc(W),\OCat_{\bc'}(W)$
into highest weight categories.

Now let us consider a more general setting. Let $\Cat$ be a $\C$-linear
abelian category equivalent to $A\operatorname{-mod}$, where $A$
is a finite dimensional associative algebra. Let $\OCat^1,\OCat^2$
be two highest weight categories with identified posets
$\Irr(\OCat^1),\Irr(\OCat^2)$, let us denote this common poset by $\Lambda$.
Assume that there are quotient functors $\pi^i:\OCat^i\twoheadrightarrow \Cat$.

{\bf General speculation}: If $\Cat$ is ``large enough'', then there is
a highest weight equivalence $\varphi:\OCat^1\rightarrow \OCat^2$
and $\pi^2\circ\varphi\cong \pi^1$.

The condition that $\pi^1,\pi^2$ are fully faithful on standardly filtered objects (below
we will say that $\pi^i$ is {\it $0$-faithful}), morally, says that
$\Cat$ is ``large enough'', but is not sufficient to establish the existence
of $\varphi$. On the other hand, what is sufficient is the following two conditions:
\begin{itemize}
\item[(L1)] $\pi^1(\Delta^1_L)=\pi^2(\Delta^2_L)$ for all $L\in \Lambda$.
\item[(L2)] For any standardly filtered objects $M,N\in \OCat^i, i=1,2$,
we have $\Ext^j_{\OCat^i}(M,N)=\Ext^j_{\Cat}(M,N)$ for $j=0,1$. In this situation
we say that $\pi^i$ is {\it $1$-faithful}.
\end{itemize}
Indeed, (L2) guarantees that $\pi^i$ restricts to an equivalence
$\OCat^{i\Delta}\xrightarrow{\sim} \Cat^{\pi^i(\Delta)}$, where $\OCat^{i\Delta}$
is the full subcategory of standardly filtered objects in $\OCat^i$ and
$\Cat^{\pi^i(\Delta)}$ denotes the full subcategory of all objects in $\Cat$
filtered by $\pi^i(\Delta_L), L\in \Lambda$. (L1) guarantees that
$\Cat^{\pi^1(\Delta)}=\Cat^{\pi^2(\Delta)}$. This gives an equivalence
$\OCat^{1\Delta}\xrightarrow{\sim}\OCat^{2\Delta}$ that uniquely
extends to an equivalence $\OCat^1\xrightarrow{\sim}\OCat^2$ with
desired properties. Unfortunately, it is almost never possible to
establish (L1) and (L2) directly.

\subsubsection{Deformations}\label{SSS:deform}
A solution found by Rouquier was to use deformations. Let $R$ be a
formal power series ring over $\C$. Suppose that we have $R$-linear
categories $\OCat^i_R, \Cat_R$ (for example, if we have an $R$-algebra
$A_R$ that is a free $R$-module specializing to $A$, then we take
$\Cat_R:=A_R\operatorname{-mod}$). Standard objects in $\OCat^i$ do not
have higher self-extensions and so uniquely deform to $\OCat^i_R$.
So do the projective objects. The deformation of the projective object
defining $\pi^i$ gives the
quotient functor $\pi^i_R:\OCat^i_R\twoheadrightarrow \Cat_R$.
Suppose that the base change $\Cat_{\operatorname{Frac}(R)}$ is
semisimple and that $\pi^i_{\operatorname{Frac}(R)}$ is an equivalence.
This gives rise to  identifications $\Irr(\OCat^i)=\Irr(\OCat^i_{\operatorname{Frac}(R)})=
\Irr(\Cat_{\operatorname{Frac}(R)})$. We further assume that the resulting
identification $\Irr(\OCat^1)\cong \Irr(\OCat^2)$ coincides with our initial
identification.

As Rouquier checked in \cite[Lemma 4.48]{rouqqsch}, an analog of (L1)
holds for $\OCat^1_R,\OCat^2_R$. Further, in \cite[Proposition 4.42]{rouqqsch}
he has checked that if $\pi^i$ is $0$-faithful, then $\pi^i_R$
is $1$-faithful. So (L1) and (L2) hold for the deformed categories,
and we get an equivalence $\varphi_R:\OCat^1_R\xrightarrow{\sim}\OCat^2_R$
that then specializes to an equivalence $\varphi:\OCat^1\rightarrow \OCat^2$
with the desired properties.

\subsubsection{Cherednik categories $\mathcal{O}$}\label{SSS_deform_cher}
Let us return to the situation when the categories of interest are
$\OCat_\bc(W),\OCat_{\bc'}(W)$. Let us explain how to deform them.
We can form the $\C[\param]$-algebra $H_{\param}$ that is the quotient
of $T(\h\oplus \h^*)\#W[\param]$ by the relations similar to above
but where we replace $c(s)\in \C$ with the basis element $c(s)\in \param^*$
corresponding to the conjugacy class of $s$ in $S$. The specialization of
$H_{\param}$ to $\bc\in \param$ coincides with $H_\bc$. We can define the category
$\OCat_{\param}(W)$ and the Verma modules there as before.

Now pick a sufficiently generic line $\ell$ through $\bc$. Let $R$ denote the completion $\C[\ell]^{\wedge_\bc}$,
this a formal power series algebra in 1 variable that is also an algebra over $\C[\param]$. We can consider the specialization $H_{R,\bc}:=R\otimes_{\C[\param]}H_{\param}$ and consider the correspoding category
$\OCat_{R,\bc}$. The algebra $R$ is naturally identified with the completion of $\C[\ell+\bc'-\bc]$
at $\bc$. So we can also form the algebra $H_{R,\bc'}$ and the category $\OCat_{R,\bc'}$.

On the other hand, we also have the $R$-algebra $\mathcal{H}_{R,{\bf q}}$. It still has the maximal
quotient of finite rank over $R$, the algebra that we denote by $\mathcal{H}^{f}_{R,{\bf q}}$.
The KZ functors extend to quotient functors $\OCat_{R,\bc}\twoheadrightarrow \mathcal{H}^{fin}_{R,{\bf q}}\operatorname{-mod}, \OCat_{R,\bc'}\twoheadrightarrow \mathcal{H}^{f}_{R,{\bf q}}\operatorname{-mod}$.
We take $\Cat_R:=\mathcal{H}^{f}_{R,{\bf q}}\operatorname{-mod}, \OCat^1_R:=\OCat_{R,\bc},
\OCat^2_R:=\OCat_{R,\bc'}$.

We need to check that the categories $\OCat^i_{\operatorname{Frac}(R),\bc}$ are semisimple.
This has to do with the choice of $\ell$ -- it is generic. Then the categories $\mathcal{O}_{\operatorname{Frac}R,\bc}(W),
\mathcal{O}_{\operatorname{Frac}R,\bc'}(W)$ are semisimple for the reasons explained in \ref{SSS_gen_ss}.
The KZ functor does not kill $\Delta_{R,\bc}(\tau)$ because this module is not torsion.
So $\pi^i_{\operatorname{Frac}(R)}$ is a category equivalence.

We need to show that the chain of identifications
$$\operatorname{Irr}(W)\xrightarrow{\sim} \operatorname{Irr}(\OCat^1_{\operatorname{Frac}(R)})
\xrightarrow{\sim} \operatorname{Irr}(\Cat_{\operatorname{Frac}(R)})
\xrightarrow{\sim} \operatorname{Irr}(\OCat^2_{\operatorname{Frac}(R)})
\xrightarrow{\sim } \operatorname{Irr}(W)$$
gives the identity. It is a general fact that it gives the bijection $\mathsf{tw}(\bc'-\bc)$,
more or less by the definition of $\mathsf{tw}$. By our choice of $\bc',\bc$, this
bijection is the identity, \cite[Theorem 7.2]{BC}.

This completes the proof of Proposition \ref{Prop:cat_equi_hw}
in the case when the functors $\mathsf{KZ}_\bc, \mathsf{KZ}_{\bc'}$
are $0$-faithful.

\subsubsection{Extended quotients}\label{SSS_ext_quot}
Let us explain how to modify the argument above  when the functors $\mathsf{KZ}_\bc, \mathsf{KZ}_{\bc'}$
are not $0$-faithful. The idea is to consider
some intermediate quotient between $\OCat_{\bc}(W)$ and $\mathcal{H}_{\bf q}^{f}(W)\operatorname{-mod}$.
This idea has appeared in \cite[Section 7]{VV_proof} in a different situation and was
used in \cite[Section 4.2]{rouq_der} to prove Proposition \ref{Prop:cat_equi_hw}.

Recall that $\mathsf{KZ}_\bc$ is the quotient functor defined by the sum of the projectives
$P_\bc(\tau)$ such that $\Supp(L_\bc(\tau))=\h$. Now consider the sum $P^1$ of all projectives
$P_\bc(\tau)$ such that $\operatorname{codim}_{\h}\Supp(L_\bc(\tau))= 1$.
Set $\tilde{P}:=P_{KZ}\oplus P^1$.   Let $\tilde{\Cat}_\bc$ denote the quotient category
of $\OCat_\bc(W)$ defined by $\tilde{P}$. The functor $\mathsf{KZ}_\bc: \OCat_\bc(W)
\twoheadrightarrow \mathcal{H}_{\bf q}^{f}(W)\operatorname{-mod}$
factorizes as $\mathsf{KZ}_\bc=\underline{\pi}_\bc\circ \overline{\pi}_\bc$,
where $\overline{\pi}_\bc$ is the quotient functor $\OCat_\bc(W)\twoheadrightarrow
\tilde{\Cat}_\bc$ defined by $\tilde{P}$ and $\underline{\pi}_\bc:
\tilde{\Cat}_\bc\twoheadrightarrow \mathcal{H}_{\bf q}^{f}(W)\operatorname{-mod}$
is the quotient functor defined by the direct summand $P_{KZ}$ of $\tilde{P}$.
We will write $\tilde{P}_\bc,\tilde{P}_{\bc'}$ when we need to indicate the
dependence on the parameter.

The following lemma was proved in \cite{VV_proof} (the proof is contained
in that of Proposition 8.1 of {\it loc.cit.}). See also \cite[Lemma 2.8]{RSVV}.

\begin{Lem}\label{Lem:0-faithful_extension}
The quotient functors $\overline{\pi}_\bc,\overline{\pi}_{\bc'}$ are $0$-faithful.
\end{Lem}

Note that the condition in 4) of Proposition \ref{Prop:KZ_properties} precisely
means that the category $\OCat_\bc(W_H)$ is semisimple for every reflection
hyperplane $H$, equivalently, there are no finite dimensional modules in this category.
Using Lemma \ref{Lem:Supp_Restr}, we see that  $\operatorname{codim}_{\h}\Supp(L_\bc(\tau))\leqslant 1$
implies $\Supp L_\bc(\tau)=\h$. So Lemma \ref{Lem:0-faithful_extension} can be regarded
as a generalization of 4) of Proposition \ref{Prop:KZ_properties}.

Now let us explain how to modify the deformation argument in order to establish
a desired equivalence $\OCat_\bc(W)\xrightarrow{\sim}\OCat_{\bc'}(W)$.
Let $R$ have the same meaning as before. We can consider the deformation $P^1_R\in \OCat_{R,\bc}(W)$
of the projective object $P^1$. So we get the deformation  $\tilde{\Cat}_{R,\bc}$ of
$\tilde{\Cat}_\bc$ that comes with quotient functors $\overline{\pi}_{R,\bc}:
\OCat_{R,\bc}(W)\twoheadrightarrow \tilde{\Cat}_{R,\bc}, \underline{\pi}_{R,\bc}:
\tilde{\Cat}_{R,\bc}\twoheadrightarrow \mathcal{H}_{R,{\bf q}}^{f}(W)$.

\begin{Prop}\label{Prop:ext_quot_equiv}
We have an equivalence $\tilde{\Cat}_{R,\bc}\xrightarrow{\sim} \tilde{\Cat}_{R,\bc'}$
that intertwines the quotient functors $\underline{\pi}_{R,\bc}$ and $\underline{\pi}_{R,\bc'}$.
\end{Prop}
\begin{proof}[Sketch of proof]
The first case to consider is when $\dim \h=1$ so that
$\tilde{\Cat}_?=\OCat_?(W)$. This is a very easy case that can be done by hand,
compare to \cite[Lemma 7.3]{VV_proof} or \cite[2.4.4]{RSVV}.

In the general case, we can argue as follows. It is sufficient to show  that
\begin{enumerate}
\item
$\mathsf{KZ}_{R,?}$ is fully faithful on $\tilde{P}_{R,?}$ (the deformation of $\tilde{P}_?$),
where $?=\bc,\bc'$.
\item $\mathsf{KZ}_{R,\bc}(\tilde{P}_{R,\bc})\cong \mathsf{KZ}_{R,\bc'}(\tilde{P}_{R,\bc'})$.
\end{enumerate}
(1) follows from (3) in Proposition \ref{Prop:KZ_properties}.
In the proof of (2) we can replace $\tilde{P}_{R,\bc}$ with $\bigoplus_H \operatorname{Ind}^{W_H}_W(P_{R,H})$,
where $P_{R,H}$ is a projective generator of $\OCat_{R,\bc}(W_H)$, these two projectives
have the same indecomposable summands (perhaps, with different multiplicities). We modify
$\tilde{P}_{R,\bc'}$ similarly. Then we use  the case of $\dim \h=1$ and Corollary \ref{Cor:KZ_Ind}.
Note that a similar argument combined with biadjointness of $\Res$ and $\Ind$
deduces (1) from the case of $\dim \h=1$ (we just need to know  that $\mathsf{KZ}_{R,\bc}$
is fully faithful on the projective objects when $\dim\h=1$).
\end{proof}

Now we can complete the proof of Proposition \ref{Prop:cat_equi_hw}
as in \ref{SSS:deform}.

\subsection{Derived equivalences}\label{SS_der_equi}
Again, we pick $\bc\in \param$ and consider the shifted lattice $\bc+\param_{\Z}$.
As we have seen in the previous section, the categories $\mathcal{O}_\bc(W)$
and $\mathcal{O}_{\bc'}(W)$ are equivalent provided $\bc,\bc'$ lie in the same chamber
and the equivalence intertwines the KZ functors. A natural question is what happens
when $\bc$ and $\bc'$ lie in different chambers. Here is the most basic version of
an answer.

\begin{Thm}\label{Thm:derived_equi}
Let $\bc,\bc'\in \param$ satisfy $\bc'-\bc\in \param_{\Z}$. Then there is a derived equivalence
$D^b(\OCat_\bc(W))\xrightarrow{\sim}D^b(\OCat_{\bc'}(W))$ that intertwine the KZ functors.
\end{Thm}
This result was conjectured in \cite{rouqqsch} and proved in \cite[Section 5]{GL}
for $W=G(\ell,1,n)$ and in \cite{rouq_der} for an arbitrary $W$. The crucial
role in the latter is played by so called wall-crossing functors.

\subsubsection{Wall-crossing functors}
Let $F$ be a face in an open chamber $C$ with $\bc\in C$. Let $C'$ be a chamber
that is opposite to $C$ with respect to $F$ (for example, if $F$ is a codimension
one face, then we take the unique other chamber adjacent to $F$). Pick
$\bc'\in (\bc+\underline{\param}_{\Z})\cap C'$. Let $\param_0:=\bc+\operatorname{Span}_{\C}(F)$, this is an affine subspace in $\param$.

We have the following important remark that follows from
Proposition \ref{Prop:cat_equi_hw}.

\begin{Rem}\label{Rem:change_c} For every
$\bc^0\in F\cap \underline{\param}_{\Z}$, we can  replace $\bc,\bc'$ with $\bc+\bc^0,
\bc'+\bc^0$
without changing the categories $\mathcal{O}$. In particular, and this is going to be very important in what follows, in studying an interplay between $\OCat_\bc,\OCat_{\bc'}$, we may assume that that $\bc$ is Zariski generic in $\param$.
\end{Rem}

The following is one of the main results of \cite{rouq_der}. It implies Theorem \ref{Thm:derived_equi}.

\begin{Prop}\label{Prop:WC_functors}
There is a derived equivalence $\WC_{\bc'\leftarrow \bc}:D^b(\OCat_\bc(W))
\rightarrow D^b(\OCat_{\bc'}(W))$ intertwining the KZ functors.
\end{Prop}
\begin{proof}[Sketch of proof]
Set $\psi:=\bc'-\bc$.   Consider the specializations $H_{\param_0+\psi}, H_{\param_0}$
of $H_{\param}$. In \cite[Section 5.2]{rouq_der}, we have produced  a $H_{\param_0+\psi}$-$H_{\param_0}$-bimodule
to be denoted by $\B_{\param_0}(\psi)$. It is {\it Harish-Chandra}  in the sense
of \cite[Section 3]{BEG1} meaning that it is finitely generated as a bimodule, and the operators
$[a,\cdot]:\B_{\param_0}(\psi)\rightarrow \B_{\param_0}(\psi)$ are locally
nilpotent for all $a\in S(\h)^W\cup S(\h^*)^W$ (note that $S(\h)^W, S(\h^*)^W$
are included into both $H_{\param_0+\psi},H_{\param_0}$ and so it makes sense
to consider the adjoint operators above). Now pick $\hat{\bc}\in \param_0$
and consider the specialization $\B_{\hat{\bc}}(\psi)$ of $\B_{\param_0}(\psi)$.
It defines a functor
\begin{equation}\label{eq:WC_functor}\B_{\hat{\bc}}(\psi)\otimes^L_{H_{\hat{\bc}}}\bullet:D^b(\OCat_{\hat{\bc}}(W))\rightarrow D^b(\OCat_{\hat{\bc}+\psi}(W)).\end{equation}
A crucial property of this functor is that if $\hat{\bc}$ is Weil generic
(=lies outside of countably many algebraic subvarieties in $\param_0$),
then (\ref{eq:WC_functor}) is a derived equivalence. Using generic flatness kind arguments
we then show that our functor is an equivalence for a Zariski generic $\hat{\bc}\in
\param_0$. Since we can achieve that $\bc$ is Zariski generic by replacing $\bc$
with $\bc+\bc^0$, this proves Proposition \ref{Prop:WC_functors}.
\end{proof}

\subsubsection{Behavior of WC on $K_0$}

Here is an easy but useful property of $\WC_{\bc'\leftarrow \bc}$,
see \cite[3.1.1]{Cher_supp}.

\begin{Lem}\label{Lem:WC_K0}
The following is true.
\begin{enumerate}
\item If $\hat{\bc}\in \param_0$ is Weil generic, then
$$\WC_{\hat{\bc}'\leftarrow \hat{\bc}}(\Delta_{\hat{\bc}}(\tau))\cong \nabla_{\hat{\bc}'}(\tau)$$
for any $\tau\in \Irr(W)$.
\item For a Zariski generic $\bc\in \param_0$, the object $\WC_{\bc'\leftarrow \bc}(\Delta_{\bc}(\tau))\in
D^b(\OCat_{\bc'}(W))$ lies in $\OCat_{\bc'}(W)$ and its class in $K_0$ coincides with that of $\Delta_{\bc'}(\tau)$.
\item In particular, $\WC_{\bc'\leftarrow \bc}$ induces the identity map on the $K_0$ groups.
\end{enumerate}
\end{Lem}

\subsubsection{Perverse equivalences}
The equivalence $\WC_{\bc'\leftarrow \bc}$ introduced above has an important property, it is
{\it perverse}. Let us explain what this means, following \cite[Section 2.6]{rouquier_ICM}.

Suppose that we have two abelian categories $\Cat^1,\Cat^2$ that are equipped
with finite filtrations by Serre subcategories: $\Cat^j=\Cat^j_0\supset
\Cat^j_1\supset\ldots \supset\Cat^j_n\supset \Cat^j_{n+1}=0$. A derived equivalence
$\varphi: D^b(\Cat^1)\rightarrow D^b(\Cat^2)$ is said to be {\it perverse} with respect to
these filtrations if the following three conditions hold:
\begin{itemize}
\item[(P1)] $\varphi$ restricts to an equivalence
between $D^b_{\Cat_i^1}(\Cat^1)$ and $D^b_{\Cat^2_i}(\Cat^2)$.
Here we write $D^b_{\Cat^j_i}(\Cat^j)$ for the full subcategory of
$D^b(\Cat^j)$ of all objects with homology in $\Cat^j_i$.
\item[(P2)] For $M\in \Cat^1_i$, we have $H_k(\varphi M)=0$
for $k<i$ and $H_{j}(\varphi M)\in \Cat^{2}_{i+1}$ for $k>i$.
\item[(P3)] The functor $M\mapsto H_i(\varphi M)$
defines an equivalence $\Cat^1_i/\Cat^1_{i+1}\xrightarrow{\sim}\Cat^2_i/\Cat^2_{i+1}$.
\end{itemize}

Let us explain how the filtrations on $\Cat^1=\OCat_\bc(W), \Cat^2=\OCat_{\bc'}(W)$.
We have chains of two-sided ideals $H_{\param_0}=I^{\param_0}_{n+1}\supset
I^{\param_0}_n\supset\ldots\supset I^{\param_0}_1\supset I^{\param_0}_0=\{0\},
H_{\param_0+\psi}=I^{\param_0+\psi}_{n+1}\supset I^{\param_0+\psi}_n\supset\ldots\supset
I^{\param_0+\psi}_1\supset I^{\param_0+\psi}_0=\{0\}$ with the following properties.
For a Weil generic $\hat{\bc}\in \param_0$, the specialization
$I^{\hat{\bc}}_i$ (resp., $I^{\hat{\bc}+\psi}_i$) is the intersection of the annihilators of all modules $M\in
\OCat_{\hat{\bc}}(W)$ (resp., $M\in \OCat_{\hat{\bc}+\psi}(W)$) with $\dim \Supp(M)<i$.
In particular, this property implies that \begin{equation}\label{eq:ideals_square}
(I^{\hat{\bc}}_i)^2=I^{\hat{\bc}}_i, (I^{\hat{\bc}+\psi}_i)^2=I^{\hat{\bc}+\psi}_i.\end{equation}

By generic flatness reasons, (\ref{eq:ideals_square}) holds for a Zariski generic $\hat{\bc}$ as well.
In particular, as before, we may assume that it holds for $\hat{\bc}=\bc$. Let $\Cat^1_{i}$
consist of all modules in $\OCat_\bc(W)$ annihilated by $I^\bc_i$, by (\ref{eq:ideals_square}) this subcategory
is closed under extensions. Define $\Cat^2_i\subset \OCat_{\bc'}(W)$ similarly.

\begin{Prop}\label{Prop:WC_perverse}
The equivalence $\WC_{\bc'\leftarrow \bc}:D^b(\OCat_\bc)\xrightarrow{\sim}D^b(\OCat_{\bc'})$
is perverse with respect to the filtrations introduced above.
\end{Prop}
\begin{proof}[Sketch of proof] First, one proves this claim for a Weil generic $\hat{\bc}$.
Then one translates (P1)-(P3) to statements involving Tor's and Ext's between $\B_\bc(\psi),
H_\bc/I^\bc_i, H^{\bc'}/I^{\bc'}_i$. One uses generic flatness arguments to establish these statements
for a Zariski generic $\bc$. This establishes (P1)-(P3) for such $\bc$.
\end{proof}

\subsubsection{Essential walls}\label{SSS_essent_wall}
Proposition \ref{Prop:WC_perverse} allows to show that some wall-crossing functors
are actually highest weight equivalences. Recall hyperplanes $\Pi_{\xi\tau}\subset \param$
introduced in \ref{SSS_chamb_walls}. We say that such a hyperplane is {\it essential} for a parameter $\bc$ if,
for a Weil generic parameter $\hat{\bc}\in \bc+\Pi_{\xi\tau}$,  the category $\OCat_{\hat{\bc}}(W)$
is not semisimple.

\begin{Cor}\label{Cor:non_essent_WC}
Let $F$ be a codimension $1$ face spanning a non-essential wall for $\bc$. Then the corresponding functor
$\WC_{\bc'\leftarrow \bc}$ restricts to a highest weight equivalence $\OCat_\bc(W)\xrightarrow{\sim}\OCat_{\bc'}(W)$.
\end{Cor}

See \cite[Lemma 3.7]{Cher_supp} for the proof.

\subsubsection{Wall-crossing and restriction functors}
Morally, wall-crossing functors intertwine the restriction and induction functors.
The following result is proved in \cite[Proposition 3.5]{Cher_supp}.

\begin{Prop}\label{Prop:WC_commut_IndRes}
Let $\bc,F,\bc',\psi,\param_0$ be as before. We assume that, for a Weil generic $\hat{\bc}\in \param_0$,
the parameters $\hat{\bc},\hat{\bc}+\psi$ lie in the opposite chambers for \underline{both} $W,\underline{W}$.
Then there are isomorphisms of functors
\begin{align*}
& \underline{\WC}_{\bc'\leftarrow \bc}\circ \Res^W_{\underline{W}}\cong \Res^W_{\underline{W}}\circ\WC_{\bc'\leftarrow \bc},\\
& \WC_{\bc'\leftarrow \bc}\circ \Ind_W^{\underline{W}}\cong \Ind_W^{\underline{W}}\circ\underline{\WC}_{\bc'\leftarrow \bc},
\end{align*}
where we write $\underline{\WC}_{\bc'\leftarrow \bc}$ for the wall-crossing functor associated to $\underline{W}$.
\end{Prop}

See \cite[3.2.3]{Cher_supp} for the proof and the discussion of the additional assumption on $\param_0,\psi$.

\subsubsection{Wall-crossing bijections}
An important feature of a perverse equivalence $\varphi:D^b(\Cat^1)\xrightarrow{\sim} D^b(\Cat^2)$
is that it induces a bijection $\Irr(\Cat^1)\xrightarrow{\sim}\Irr(\Cat^2)$.
We have the bijection $\Irr(\Cat^1_i/\Cat^1_{i+1})\xrightarrow{\sim}\Irr(\Cat^2_i/\Cat^2_{i+1})$
given by the equivalence in (P3) for any $i$. These bijections constitute a desired
bijection $\Irr(\Cat^1)\xrightarrow{\sim}
\Irr(\Cat^2)$. For the wall-crossing functor $\WC_{\bc'\leftarrow \bc}$ the corresponding bijection
$\Irr(W)\rightarrow \Irr(W)$ will be called the {\it wall-crossing} bijection and denoted
by $\wc_{\bc'\rightarrow \bc}$.

We have the following two important properties of wall-crossing  bijections
established in \cite{Cher_supp}. The first property is a consequence of the
construction of a wall-crossing functor as the derived tensor product with a
Harish-Chandra bimodule.

\begin{Lem}\label{Lem:wc_supp}
The wall-crossing bijections preserve supports.
\end{Lem}
This is \cite[Corollary 2.13]{Cher_supp}.

\begin{Prop}\label{Prop:wc_bij_indep}
The wall-crossing bijection $\wc_{\bc+\psi\leftarrow \bc}$ is independent of the choice
of a Zariski generic $\bc\in \param_0$.
\end{Prop}
This is \cite[Proposition 3.1]{Cher_supp}.

The second property allows to reduce the computation of the wall-crossing bijection
to the case when $\bc$ is Weil generic in $\param_0$. For such a $\bc$, the categories
$\OCat_\bc(W),\OCat_{\bc+\psi}(W)$ simplify and it is easier to compute the wall-crossing
bijections there. We will do some explicit computations in the case of $W=G(\ell,1,n)$
in Section \ref{SS_supp_comb_comput}.

\section{Cyclotomic categories $\mathcal{O}$ and categorification}\label{S_cyclot}
\subsection{Cyclotomic categories $\mathcal{O}$ and Fock spaces}
\subsubsection{Recap}\label{SSS_cyclot_Recap}
From now on we are interested in the groups $W=G(\ell,1,n)$. Recall that
that the set $\Irr(W)$ is identified with the set $\mathcal{P}_\ell(n)$
of the $\ell$-multipartitions of $n$, see Example \ref{Ex:cyclot_irrep}.

We will write $\OCat_\bc(n)$ for $\OCat_\bc(G(\ell,1,n))$.

Recall the parameters $\kappa, h_0,\ldots h_{\ell-1}$, \ref{SSS_Euler}. In the case when $\kappa=0$,
we get $H_\bc=H_\bc(\mu_\ell)^{\otimes n}\# \Sym_n$ and so the category $\OCat_\bc(n)$
coincides with the category of $\Sym_n$-equivariant objects in $\OCat_\bc(1)^{\otimes n}$.
This is an easy case and we are not going to consider it below. So we assume that $\kappa\neq 0$.
In this case we introduce another set of parameters, $s_0,\ldots,s_{\ell-1},$ by
$h_i=\kappa s_i-i/\ell$.  We write ${\bf s}$ for the collection $(s_0,\ldots,s_{\ell-1})$
(note that $s_0,\ldots,s_{\ell-1}$ are defined up to a common summand).
We often write $\OCat_{\kappa,{\bf s}}(n)$ instead of $\OCat_\bc(n)$.

Also recall that the Hecke algebra $\mathcal{H}_{\bf q}(W)$ coincides with the cyclotomic
Hecke algebra, $\mathcal{H}_{q,{\bf s}}(n)$ with parameters $q=\exp(2\pi\sqrt{-1}\kappa)$
and $Q_i:=\exp(2\pi\sqrt{-1}\kappa s_i)$, Example \ref{Ex:cyclot_Hecke}.

\subsubsection{Order on $\mathcal{O}_\bc$}\label{SSS_O_order}
Recall, \ref{SSS_Euler}, that to a box $b$ with coordinates $(x,y)$ in the partition
$\lambda^{(i)}$ we assign the number $c_b=\kappa\ell(x-y)+\ell h_i=\kappa \ell(x-y+s_i)-i$.
We will write $\cont^{\bf s}(b)$ for $x-y+s_i$. We can take the order $\leqslant_\bc$ defined
by $\lambda\leqslant_\bc \mu$ if $\lambda=\mu$ or $\sum_{b\in \lambda} c_b-\sum_{b\in \mu} c_b\in \Z_{>0}$
for a highest weight order on $\OCat_\bc(W)$. It turns out that a rougher order will also work.

Define an equivalence relation on boxes by $b\sim b'$ if $\kappa (\cont^{\bf s}(b)-\cont^{\bf s}(b'))\in \Z$.
We write $b\leqslant b'$ if $b\sim b'$ and $c_b-c_{b'}\in \Z_{\leqslant 0}$. Define the order
$\preceq_\bc$ on $\mathcal{P}_\ell(n)$ as follows: we set $\lambda\preceq_\bc \lambda'$ if one can order
boxes $b_1,\ldots,b_n$ of $\lambda$ and $b_1',\ldots,b_n'$ of $\lambda'$ so that $b_i\leqslant b'_i$
for any $i$. Clearly, $\lambda\preceq_\bc \lambda'$ implies $\lambda\leqslant _\bc \lambda'$.

The following result is due to Dunkl and Griffeth, \cite[Theorem 1.2]{DG}.

\begin{Prop}\label{Prop:better order}
One can take $\preceq_\bc$ for a highest weight order for $\OCat_\bc(n)$.
\end{Prop}

\subsubsection{Fock space}
We set $\OCat_{\kappa,{\bf s}}:=\bigoplus_{n\geqslant 0}\OCat_{\kappa,{\bf s}}(n)$.
We have the basis $|\lambda\rangle:=[\Delta_{\kappa,{\bf s}}(\lambda)]$ in
$K_0^{\C}(\OCat_{\kappa,{\bf s}})$ indexed by
$\lambda\in \mathcal{P}_\ell:=\bigsqcup_{n\geqslant 0}\mathcal{P}_\ell(n)$. In other words,
$K_0^{\C}(\OCat_{\kappa,{\bf s}})$ is the level $\ell$ Fock space.

\subsubsection{Decomposition}\label{SSS_decomp1}
For certain values of ${\bf s}$, the category $\mathcal{O}_{\kappa,{\bf s}}$
can be decomposed into the product of categories $\mathcal{O}_{\kappa,?}$ for smaller
$\ell$.

We define an equivalence relation $\sim_{\bc}$ on $\{0,\ldots,\ell-1\}$ by setting
$i\sim_\bc j$ if the $i$th and $j$th partitions can contain equivalent boxes,
i.e., $s_i-s_j\in \kappa^{-1}\Z+\Z$. For an equivalence class $\alpha$, we write
${\bf s}(\alpha)$ for $(s_i)_{i\in \alpha}$ and  $\mathcal{P}_\alpha$
for the subset of all $\lambda\in \mathcal{P}_\ell$ with $\lambda^{(j)}=\varnothing$
for $j\not\in \alpha$. Form the category $\bigotimes_\alpha \OCat_{\kappa,{\bf s}(\alpha)}$.
The simples in this category are labelled by the set $\prod_\alpha \mathcal{P}_\alpha$
that is naturally identified with $\mathcal{P}_\ell$.

\begin{Prop}\label{Prop:cat_O_decomp}
There is a highest weight equivalence $\OCat_{\kappa,{\bf s}}\xrightarrow{\sim} \bigotimes_\alpha
\OCat_{\kappa,{\bf s}(\alpha)}$.
\end{Prop}

In the case when the assumptions of (4) of Proposition \ref{Prop:KZ_properties}
are satisfied, this result was established in \cite[Section 6]{rouqqsch}.
In general, it can be proved using the extension of Rouquier's approach explained in
\ref{SSS_ext_quot}.

Using Proposition \ref{Prop:cat_O_decomp}, we can reduce the study of categories
$\mathcal{O}_{\kappa,{\bf s}}$ to the case when we have just one equivalence
class in $\{0,1,\ldots,\ell-1\}$.

\subsubsection{Essential walls}\label{SSS_essent_walls_cyclot}
The definition of an essential wall was given in \ref{SSS_essent_wall}.
Here we are going to describe the essential walls for the groups $G(\ell,1,n)$.
According to  \cite[Lemma 5.5]{Cher_supp}, we have the following essential walls
\begin{enumerate}
\item $\kappa=0$ for the parameters $\bc$, where the $\kappa$-component is a rational
number with denominator between $2$ and $n$.
\item $h_i-h_j=\kappa m$ with $i\neq j$ and $|m|<n$ -- for the parameters
$\bc$ satisfying $s_i-s_j-m\in \kappa^{-1}\Z$.
\end{enumerate}

These walls split $\bc+\param_{\Z}$ into the union of polyhedral chambers
to be called {\it essential chambers}. Recall, Corollary \ref{Cor:non_essent_WC},
that the categories $\mathcal{O}$ with parameters in one essential chamber are
highest weight equivalent.

\subsection{Categorical Kac-Moody action}
\subsubsection{Kac-Moody action on the Fock space}\label{SSS_Fock}
Let $\kappa,{\bf s}$ be the same as in \ref{SSS_cyclot_Recap}.

Define the Kac-Moody algebra $\g_\kappa$ as $\hat{\slf}_e$ if $\kappa$
is rational with denominator $e$ and as $\slf_\infty$ if $\kappa$
is irrational. Here $\slf_\infty$ stands for the Kac-Moody algebra of infinite
rank associated to the type A Dynkin diagram that is infinite in both direction.
When $e=1$, we assume that $\g_\kappa=\{0\}$.

Define the Kac-Moody algebra
$\g_{\kappa,{\bf s}}$ as the product of several copies of $\g_\kappa$, one
per equivalence class for $\sim_\bc$ in $\{0,1,\ldots,\ell-1\}$. This algebra
has generators $e_z, f_z$, where $z$ runs over the subset in $\C/\kappa^{-1}\Z$ of the
elements of the form
$s_i+m$, where $m$ is an integer. We are going to define an action of
$\g_{\kappa,{\bf s}}$ on $\mathcal{F}^\ell$, the Fock space of level $\ell$.

The action is defined as follows. We say that a box $b$ is a $z$-box if
$\cont^{\bf s}(b)$ equals to $z$ in $\C/\kappa^{-1}\Z$. We set
$f_z|\lambda\rangle:=\sum_{\mu}|\mu\rangle$, where the sum is
taken over all $\ell$-partitions $\mu$ that are obtained from
$\lambda$ by adding a $z$-box. Similarly, we set
$e_z|\lambda\rangle=\sum_{\nu}|\nu\rangle$, where the sum is
taken over all $\nu$ obtained from $\lambda$ by removing a
$z$-box.

We write $\mathcal{F}_{\kappa,\bf s}$ for the space $\mathcal{F}^\ell$ equipped
with this $\g_{\kappa}$-action. We note that we have a natural isomorphism
of $\g_\kappa$-modules, $\mathcal{F}_{\kappa,\bf s}=\bigotimes_{i=0}^{\ell-1} \mathcal{F}_{\kappa, s_i}$.

\subsubsection{Type A categorical Kac-Moody action}
Let $\Cat$ be an abelian $\C$-linear category, where all objects have finite length.
A {\it type A categorical Kac-Moody action} on $\Cat$ as defined in
\cite[5.3.7, 5.3.8]{Rouquier_2Kac} consists of the following data:
\begin{enumerate}
\item exact endo-functors $E,F$ of $\Cat$ and a number $q\in \C\setminus \{0,1\}$,
\item adjointness morphisms $\mathsf{1}\rightarrow EF, FE\rightarrow \mathsf{1}$,
\item endomorphisms $X\in \End(E), T\in \End(E^2)$.
\end{enumerate}
These data are supposed to satisfy the axioms to be provided below.
We will need the following notation.
Let $I$ be a subset in $\C^\times$. Define a Kac-Moody algebra $\g_I$
as follows. Define an unoriented graph structure on $I$ by
connecting $z$ and $z'$ if $z'z^{-1}=q^{\pm 1}$. Then $\g_I$ is the Kac-Moody
algebra defined from $I$, it is the product of several copies of $\hat{\slf}_e$
if $q$ is a primitive root of unity of order $e$, and is the product of several
copies of $\slf_\infty$ if $q$ is not a root of unity. For example,
taking $q=\exp(2\pi\sqrt{-1}\kappa)$ and $I=\{Q_0,\ldots,Q_{\ell-1}\}$, we get $\g_I=\g_{\kappa,{\bf s}}$.

Here are the axioms of a categorical action.

\begin{itemize}
\item[(i)] $F$ is isomorphic to the left adjoint of $E$.
\item[(ii)] For any $d$, the map  $X_i\mapsto {\bf 1}^{i-1}X{\bf 1}^{d-i},
T_i\mapsto {\bf 1}^{i-1}T{\bf 1}^{d-1-i}$ extends to a homomorphism
$\mathcal{H}_q^{aff}(d)\rightarrow \End(E^d)$.
\item[(iii)] Let $E=\bigoplus_{z\in \C}E_z$ be the decomposition into eigen-functors according to $X$,
and $F=\bigoplus_{z\in \C}F_z$ be the decomposition coming from (2).
The operators $[E_z],[F_z]$ give rise to an integrable representation of $\g_I$ on $K_0^{\C}(\Cat)$,
where $I:=\{z\in \C^\times| E_z\neq 0\}$.
\item[(iv)] Let $\Cat_\nu$ denote the Serre subcategory of $\Cat$ spanned by the simples $L$ with
$[L]\in K_0^{\C}(\Cat_\nu)$, where $K_0^{\C}(\Cat_\nu)$ is the $\nu$ weight space for the
$\g_I$-module $K_0^{\C}(\Cat)$. Then $\Cat=\bigoplus_\nu \Cat_\nu$.
\end{itemize}

\subsubsection{Example: cyclotomic Hecke algebras}
Let us provide an example that first appeared in \cite[Section 7.2]{CR}.
Let us write $\mathcal{H}^{\bf s}_q(n)$
for the cyclotomic Hecke algebra as in \ref{SSS_cyclot_Recap}.
Set $\Cat:=\bigoplus_{n\geqslant 0}\mathcal{H}^{\bf s}_q(n)\operatorname{-mod}$.

Let us explain the categorification data in this case. We take the same $q$
as the eponymous parameter for our Hecke algebra and
$I:=\{Q_0,\ldots,Q_{\ell-1}\}$.  Let $\,^\mathcal{H}\Res_{n-1}^n$
denote the restriction functor $\mathcal{H}^{\bf s}_q(n)\operatorname{-mod}
\rightarrow \mathcal{H}^{\bf s}_q(n-1)\operatorname{-mod}$ (we set
$\,^\mathcal{H}\Res_{-1}^0=0$) and let $\,^\mathcal{H}\operatorname{Coind}^{n-1}_n$
denote the coinduction functor, the right adjoint of
the restriction functor. We set $E:=\bigoplus_{n=0}^\infty \,^\mathcal{H}\Res_{n-1}^n$
and $F:=\bigoplus_{n=0}^\infty \,^\mathcal{H}\operatorname{CoInd}_{n+1}^n$ so that $F$ is naturally
identified with the right adjoint of $E$.

The endomorphism $X$ on the summand $\,^\mathcal{H}\Res_{n-1}^n$ is given by
the multiplication by $X_n\in \mathcal{H}^{\bf s}_q(n)$. Similarly,
on the summand $\,^\mathcal{H}\Res_{n-2}^n$ of $E^2=\bigoplus_{n\geqslant 0}
\,^\mathcal{H}\Res_{n-2}^n$ we define $T$ as the multiplication by $T_{n-1}$.

It is shown in \cite[Section 13.6]{Ariki} that $K_0^{\C}(\Cat)$ is the irreducible highest
weight $\g_I$-module $L(\omega_{\bf s})$ whose highest weight equals $\omega_{\bf s}:=\sum_{z\in I}n_z \omega_z$,
where $\omega_z$ is the fundamental weight in the vertex $z$ and $n_z$
is the number $s_i$ such that $\exp(2\pi\sqrt{-1}\kappa s_i)=z$.
The category $\Cat$ admits a weight decomposition as in axiom (iv).
And so it indeed carries a categorical $\g_I$-action.


\subsubsection{Categorical Kac-Moody action on $\mathcal{O}_\bc$}\label{SSS_cat_KM_O}
Now let us proceed to defining a categorical $\g_\bc$-action on $\mathcal{O}_\bc$
that categorifies the $\g_\bc$-action on $\mathcal{F}^\ell$ from \ref{SSS_Fock}.
Here we  follow \cite{Shan}.

Consider the functors $\,^{\OCat}\Res^n_{n-1}:=\,^{\OCat}\Res_{G(\ell,1,n-1)}^{G(\ell,1,n)}$
and $\,^{\OCat}\Ind^n_{n+1}:=\,^{\OCat}\Ind_{G(\ell,1,n+1)}^{G(\ell,1,n)}$. We set
$E:=\bigoplus_{n=0}^\infty \,^{\OCat}\Res^n_{n-1}$ and
$F:=\bigoplus_{n=0}^\infty \,^{\OCat}\Ind^n_{n+1}$. This functor
$F$ is naturally the right adjoint of $E$ but it is also isomorphic
to the left adjoint of $E$ so we have (i).

Now let us explain how to construct the endomorphisms  $X\in \End(E),T\in \End(E^2)$.
We have the KZ functor $\KZ_\bc:\OCat_\bc\twoheadrightarrow \bigoplus_{n=0}^\infty
\mathcal{H}^{\bf s}_q(n)\operatorname{-mod}$, the sum of KZ functors
$\OCat_\bc(n)\twoheadrightarrow \mathcal{H}^{\bf s}_q(n)\operatorname{-mod}$.
By Proposition \ref{Prop:KZ_Res},
\begin{equation}\label{eq:KZ_E_commute} \,^{\mathcal{H}}E\circ \operatorname{KZ}_\bc\cong \operatorname{KZ}_\bc\circ \,^{\mathcal{O}}E.\end{equation}
Since $\KZ_\bc$ is fully faithful on the projectives objects, see (3) of Proposition
\ref{Prop:KZ_properties}, we deduce that isomorphism (\ref{eq:KZ_E_commute})
induces an isomorphism $\End(\,^{\mathcal{H}}E)\cong \End(\,^{\mathcal{O}}E)$
that gives as an element $X$ in the right hand side.
Similarly, we have an isomorphism   $\End(\,^{\mathcal{H}}E^2)\cong \End(\,^{\mathcal{O}}E^2)$
that gives us $T\in \End(\,^{\mathcal{O}}E^2)$.
(ii) follows.

Let us explain how to prove (iii) and (iv) (the same techniques work for cyclotomic
Hecke algebras). First, consider the case when the parameter $\hat{\bc}$ is Weil generic.
Here the category $\mathcal{O}_{\hat{\bc}}$ is semisimple and $\operatorname{KZ}_{\hat{\bc}}$ is an
equivalence. We have $E\Delta_{\hat{\bc}}(\lambda)\cong \bigoplus_{\nu} \Delta_{\hat{\bc}}(\nu)$, where
the sum is taken over all $\nu$ obtained from $\lambda$ by removing a box. The endomorphism
$X$ acts on $\Delta_{\hat{\bc}}(\nu)$ by $z(\hat{\bc},\lambda\setminus \nu):=q^{x-y}Q_i$,
where $\lambda\setminus \nu=(x,y,i)$, see \cite[Section 13.6]{Ariki}.
Now take an arbitrary $\bc$ and pick a generic line $\mathcal{L}$ through $\bc$ so that $\OCat_{\hat{\bc}}$
is semisimple for any $\hat{\bc}\in\ell\setminus \{\bc\}$. The class $[E_z \Delta_\bc(\lambda)]$
coincides with $\sum_\nu  [\Delta_{\hat{\bc}}(\nu)]$, where sum is taken over
all $\nu$ such that the number $z(\hat{\bc},\lambda\setminus \nu)$ approaches $z$
as $\hat{\bc}\in \mathcal{L}$ approaches $\bc$. It follows that $[E_z\Delta_\bc(\lambda)]=\sum_{\nu}[\Delta_\bc(\nu)]$,
where the summation is taken over all $\nu$ such that $\nu\subset \lambda,\lambda\setminus \nu$  is a $z$-box.
Repeating this argument for $F$ and the object $\nabla_\bc(\lambda)$, we get
$[F_z\nabla_\bc(\lambda)]=\sum_{\mu}[\nabla_\bc(\mu)]$, where the sum is taken
over all $\mu$ such that $\lambda\subset \mu$ and  $\mu\setminus \lambda$
is a $z$-box. Since $[\nabla_\bc(\lambda)]=[\Delta_\bc(\lambda)]$, we see that the operators
$[E_z],[F_z]$ define the action of $\g_\bc$ on $\mathcal{F}^\ell$ considered
in \ref{SSS_Fock}. This establishes (iii).

(iv) is shown in a similar fashion. For a multiset $A$ of $\C/\kappa^{-1}\Z$
define the subspace $\mathcal{F}^{\ell}(A)\subset \mathcal{F}^\ell$ (resp.,
the subcategory $\OCat_{\bc}(A)$) to be  the linear span of $|\lambda\rangle$
(resp., the Serre span of $\Delta_\bc(\lambda)$)  with $\{\cont^{\bf s}(\lambda)\}=A$.
We consider the central subalgebras $\C[X_1^{\pm 1},\ldots,X_n^{\pm n}]^{\Sym_n}
\subset \mathcal{H}_q^{aff}(n)$. Since $\operatorname{KZ}_\bc$ is fully faithful on
the projective objects, we see that $\C[X_1^{\pm 1},\ldots,X_n^{\pm n}]^{\Sym_n}$  maps into
the endomorphism algebra of the identity functor of $\OCat_\bc(n)$ and so
functorially acts on objects of $\OCat_\bc(n)$. A multiset $A$ of cardinality
$n$ defines a character $\chi_A:\C[X_1^{\pm 1},\ldots,X_n^{\pm n}]^{\Sym_n}$
that sends $f$ to $f(\exp(2\pi\sqrt{-1}a_1),\ldots, \exp(2\pi\sqrt{-1}a_n)),$
where $A=\{a_1,\ldots,a_n\}$. A degeneration argument similar to the previous
paragraph shows that $\OCat_{\bc}(A)$ consists precisely of the objects
$M\in \OCat_{\bc}(n)$ (where $n=|A|$) such that $\C[X_1^{\pm 1},\ldots, X_n^{\pm 1}]^{\Sym_n}$
acts on $M$ with generalized eigen-character $\chi_A$. From here it follows
that $K_0^{\C}(\OCat_\bc(A))=\mathcal{F}^\ell(A)$.

\subsubsection{Compatibility with highest weight structure}
The categorical $\g_\bc$-action on $\OCat_\bc$ is highest weight in the sense
of \cite{cryst,hw_str}. The definition from \cite{hw_str} is quite
technical, let us explain what it boils down to in our case.
The following result is a direct consequence of \ref{SSS_cat_KM_O}
and Proposition \ref{Prop:ResInd_dual}.

\begin{Lem}\label{Lem:hw_compat}
The object $E_z \Delta_\bc(\lambda)$ has a filtration by $\Delta_\bc(\nu)$,
where $\nu$ runs over all $\ell$-partitions contained in $\lambda$
such that $\lambda\setminus \nu$ is a $z$-box, each $\Delta_\bc(\nu)$
occurs with multiplicity $1$. Similarly,
$F_z \Delta_\bc(\lambda)$ has a filtration by $\Delta_\bc(\mu)$,
where $\mu$ runs over all $\ell$-partitions contained in $\lambda$
such that $\mu\setminus \lambda$ is a $z$-box, each $\Delta_\bc(\mu)$
occurs with multiplicity $1$.
\end{Lem}

\subsection{Crystal}
\subsubsection{Crystal for a categorical action}
Let us recall the definition of a crystal corresponding to a categorical $\g_I$-action.
The crystal structure will be defined for each $z\in I$ separately
so we will get $I$ copies of an $\slf_2$-crystal.

In this paper, by an $\slf_2$-crystal we mean a set $C$ with maps
$\tilde{e},\tilde{f}:C\rightarrow C\sqcup \{0\}$ such that the following holds.
\begin{itemize}
\item For any $v\in C$ there are $m$ and $n$ such that $\tilde{e}^n v=\tilde{f}^m v=0$.
\item Moreover, for $u,v\in C$,  the equalities $\tilde{e}u=v$ and
$\tilde{f}v=u$ are equivalent.
\end{itemize}

Now let $\Cat$ be a $\g_I$-categorification. We will introduce
a $\g_I$-crystal structure on the set $\Irr(\Cat)$.
Namely, pick $L\in \Irr(\Cat)$ and consider the object $E_zL$.
If it is nonzero, it has simple
head (=maximal semisimple quotient) and simple socle
(=maximal semisimple sub) and those two are isomorphic, \cite[Proposition 5.20]{CR}.
We take that simple object for $\tilde{e}_z L$ if
$E_zL\neq 0$. We set $\tilde{e}_zL=0$ if $E_z L=0$.
We define $\tilde{f}_z L$ similarly. That we get a $\g_I$-crystal
follows from \cite[Proposition 5.20]{CR} combined with
\cite[Section 5]{BK}.

\subsubsection{Computation for $\mathcal{O}_\bc$}
Now let us explain how to compute the crystal on $\mathcal{P}_\ell=\Irr(\OCat_\bc)$,
which was done in \cite{cryst}. In order to compute $\tilde{e}_z\lambda, \tilde{f}_z\lambda$,
we first record the {\it $z$-signature} of $\lambda$ that is a sequence of
$+$'s and $-$'s. Then we perform a certain reduction procedure getting what
we call the {\it reduced $z$-signature}. Based on that signature, we can then
compute $\tilde{e}_z\lambda, \tilde{f}_z\lambda$. What we get is
a crystal on $\mathcal{P}_\ell$ very similar to what was discovered by
Uglov in \cite{Uglov}.

Let us start by constructing the $z$-signature of $\lambda$. Let us take the addable
and removable boxes of $\lambda$ and order them in a decreasing way
according to the order $\prec_\bc$ introduced above. Here is an example
when $\ell=2, \kappa=-1/2, s_0=0, s_1=-1$, we take $z$ to be $0$ modulo $2$
and $\lambda^{(0)}=(2,2), \lambda^{(1)}=(3,1^3)$. We get box $b_1,\ldots,b_5$.

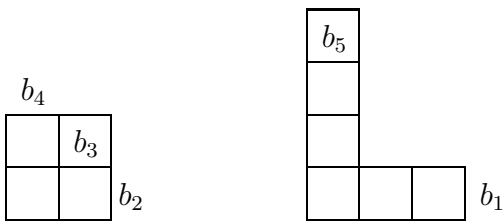
\begin{picture}(120,30)
\put(2,1){\line(0,1){14}}
\put(9,1){\line(0,1){14}}
\put(16,1){\line(0,1){14}}
\put(2,1){\line(1,0){14}}
\put(2,8){\line(1,0){14}}
\put(2,15){\line(1,0){14}}
\put(42,1){\line(0,1){28}}
\put(49,1){\line(0,1){28}}
\put(56,1){\line(0,1){7}}
\put(63,1){\line(0,1){7}}
\put(42,1){\line(1,0){21}}
\put(42,8){\line(1,0){21}}
\put(42,15){\line(1,0){7}}
\put(42,22){\line(1,0){7}}
\put(42,29){\line(1,0){7}}
\put(65,3){$b_1$}
\put(17,3){$b_2$}
\put(11,10){$b_3$}
\put(4,17){$b_4$}
\put(44,24){$b_5$}
\end{picture}

To get the signature we write a $+$ for each addable box and a $-$
for each removable box. In the example above, we get $++-+-$.

Now we are going to reduce the signature using the following steps.
Initially, we have a sequence of $+$'s and $-$'s. We erase consecutive
$-+$ leaving empty spaces. On each next step, if we have $-,+$ in
this order separated by empty spaces, we erase them. We continue until
there is no $-$'s to the left of a $+$. What we get is the {\it reduced
signature}. It is easy to see that it does not depend on the order in
which we perform our steps. In the example above, we get $++-$.

To define $\tilde{e}_z\lambda$ we pick the leftmost $-$ in the reduced
$z$-signature of $\lambda$ and remove the box in the corresponding
position of $\lambda$ (in the reduced signature, we replace this $-$
with a $+$). If the reduced $z$-signature of $\lambda$ consists only of
$+$'s, we set $\tilde{e}_z \lambda=0$. To define $\tilde{f}_z\lambda$
we pick the rightmost $+$ in the reduced signature of $\lambda$
and add the corresponding box to $\lambda$. If there are no $+$'s in
the reduced signature we set $\tilde{f}_z\lambda=0$.

In the example above we remove the box $b_5$ to get $\tilde{e}_z\lambda=((2,2), (3,1^2))$
and add the box  $b_2$ to get  $\tilde{f}_z\lambda=((3,2),(3,1^3))$.

\subsubsection{Special cases}
Here we are going to give a partial description of  several crystals that are going to appear below.

\begin{Ex}\label{Ex:level1_cryst}
Consider the case $\ell=1, \kappa<0$ and let $s_0=0$.
Then in the $i$-signature of $\lambda\in  \mathcal{P}_1$ we
include all addable/removable boxes $(x,y)$ with $x-y$
congruent to $i$ modulo $e$.

We note that in this case all connected components are isomorphic
to the component of $\varnothing$ via an explicit isomorphism.
Namely, the singular vertices are precisely the partitions
divisible by $e$ (meaning that all parts are divisible by $e$).
For such a partition $\mu$, an isomorphism
$\mathcal{P}_{1}(\varnothing)\xrightarrow{\sim} \mathcal{P}_1(\mu)$
is given by $\lambda\rightarrow \lambda+\mu$ (the sum is also taken
component-wise).

The component of $\varnothing$ consists of all partitions, where
each column occurs less than $e$ times.
\end{Ex}

\begin{Ex}\label{Ex:level2_cryst_1}
Now consider the case when $\ell=2$, and $\kappa$ is irrational.
Suppose that $s_1=0, s_2=m$, where $m$ is an integer. A box
$b=(x,y,1)$ is a $\kappa^i$-box if and only if $x-y=i$.
A box $b=(x,y,2)$ is a $\kappa^i$-box if and only if $x-y=i-m$.
Note that a $\kappa^i$-box from the first partition is always
smaller than an $i$-box from the second partition. The $\kappa^i$-signature
of any $\lambda$ consists of at most two elements.

Let us describe the singular bi-partitions $\lambda$ in this case. Those are the
partitions, where all signatures consist of $+$'s only or are equal to
$-+$. This implies, in particular, that $\lambda^{(1)}=\varnothing$.
The only removable box in $\lambda^{(2)}$ is a $1$-box so it has content
$m$. So the singular bi-partitions are the bipartitions of the form
$(\varnothing, ((k+m)^k))$, where $k$ is a positive integer. Here and below,
in the notation for a partition, a superscript means the multiplicity.
\end{Ex}

\begin{Ex}\label{Ex:level2_cryst_2}
We still consider the case $\ell=2, \kappa$ is irrational.
But now take $s_1=0, s_2=m-\kappa^{-1}$. The description
of the $\kappa^i$-boxes is the same but the order is reversed:
a $\kappa^i$-box from the first partition is always
bigger than a $\kappa^i$-box from the second partition.
So the singular bi-partitions are precisely the
bi-partitions of the form $((k^{m+k}),\varnothing)$.
\end{Ex}

\subsection{Type A category $\mathcal{O}$ and $q$-Schur algebra}
In order to proceed to a categorical Heisenberg action, we need to establish an equivalence
of categories $\mathcal{O}_c(\Sym_n)$ and the categories of modules over $q$-Schur algebras.
This equivalence was established by Rouquier in \cite{rouqqsch} when $c\not\in 1/2+\Z$
using the approach described  in Section \ref{SS_abelian_equi}. An alternative approach
based on induction and restriction functors to be explained in this section was found in \cite[Appendix]{VV_proof}.

\subsubsection{$q$-Schur algebras} Fix $n$ and consider $N\geqslant n$. We can form Lusztig's
form (i.e., the form with divided powers)
of the quantum group $U_{\epsilon}(\gl_N)$, where $\epsilon:=\exp(\pi \sqrt{-1}\kappa)$. It makes sense
to speak about polynomial representations of $U_{\epsilon}(\gl_N)$ of degree $n$, those are the modules
where all weights $(w_1,\ldots,w_m)$ satisfy
\begin{equation}\label{eq:polyn_wts} w_N\geqslant 0, w_1+\ldots+w_N=n.\end{equation} The action
of $U_\epsilon(\gl_m)$ on any such module factors through the {\it Schur algebra} $\mathcal{S}_\epsilon(N,n)$
that, by definition, is the image of $U_\epsilon(\gl_N)$ in $\End(V^{\otimes n})$, where $V$
stands for the tautological $U_\epsilon(\gl_N)$-module $\C^N$.

The category of finite dimensional $U_\epsilon(\gl_N)$-modules is known to be highest weight,
where the standard objects are the Weyl modules, and the order is, for example,
the dominance ordering on highest weights (one can also take rougher orders).
Recall that the dominance ordering is defined by $w\leqslant w'$ if
$\sum_{i=1}^m w_i=\sum_{i=1}^m   w_i'$ and $\sum_{i=1}^j w_i\leqslant
\sum_{i=1}^j w_i'$.

It follows that  the weights satisfying (\ref{eq:polyn_wts})
form a poset ideal. So $\mathcal{S}_\epsilon(N,n)$ is a highest weight
subcategory of the category of finite dimensional $U_\epsilon(\gl_N)$-modules.
Also it is easy to see that, for $N\geqslant n$, we have a highest weight
equivalence $\mathcal{S}_\epsilon(N,n)\operatorname{-mod}\xrightarrow{\sim}
\mathcal{S}_\epsilon(N+1,n)\operatorname{-mod}$. The standard object labelled
by a partition $\lambda$ in $\mathcal{S}_\epsilon(N,n)\operatorname{-mod}$
will be denoted by $\Delta^S(\lambda)$.

%
%

There is an alternative definition of the $q$-Schur algebras (up to a Morita
equivalence), see \cite{DJ}. An equivalence with the previous definition is explained
in \cite[Section 6]{Martin}.

Consider the  one-dimensional $\mathcal{H}_q(n)$-module, where
the generator $T_i$ of $\mathcal{H}_q(n)$ acts by $q$. We denote this module by
by $\mathsf{Triv}_{n}$. For $\lambda\in \mathcal{P}(n)$, set
$$IT(\lambda):=\,^\mathcal{H}\Ind^{\Sym_\lambda}_{\Sym_{n}}(\operatorname{Triv}_{\lambda_1}
\boxtimes\ldots\boxtimes\operatorname{Triv}_{\lambda_k}).$$
Here $\lambda=(\lambda_1,\ldots,\lambda_k)$ and we write $\Sym_\lambda$ for the parabolic
subgroup $\Sym_{\lambda_1}\times\ldots\times \Sym_{\lambda_k}\subset \Sym_n$.
Then the  $S_\epsilon(N,n)\operatorname{-mod}$ is equivalent to
the category of right modules over $$\operatorname{End}(\bigoplus_{\lambda\in \mathcal{P}(n)}IT(\lambda)).$$
In other words, the category of indecomposable projectives in $S_\epsilon(N,n)\operatorname{-mod}$
is identified with the category of {\it Young modules} $Y(\lambda)$. By definition,
$Y(\lambda)$ is the unique indecomposable summand of $IT(\lambda)$ that does not
occur in $IT(\lambda')$ for $\lambda'>\lambda$. The projective cover
$P^S(\lambda)$ of $\Delta^S(\lambda)$ corresponds to $Y(\lambda)$.

Note that $\mathcal{H}_q(n)=IT((1^n))$. It follows that we get
a quotient functor $\operatorname{Sh}:S_\epsilon(N,n)\operatorname{-mod}
\twoheadrightarrow \mathcal{H}_q(n)\operatorname{-mod}$ called the Schur functor.
The reason is that, in the quantum group realization of $S_\epsilon(N,n)\operatorname{-mod}$,
this functor is given by $\Hom_{S_\epsilon(N,n)}(V^{\otimes n},\bullet)$,
where $V$ is the tautological $U_\epsilon(\gl_N)$-module. By the second
construction of $S_\epsilon(N,n)\operatorname{-mod}$ this functor is fully
faithful on the projectives.

%

\subsubsection{Equivalence theorem}
Here we are going to prove the following theorem.

\begin{Prop}\label{Prop:equivalence_typeA}
Let $c>0$. There is a highest weight equivalence $\varphi:\OCat_c(n)\xrightarrow{\sim}S_\epsilon(N,n)\operatorname{-mod}$
that intertwines $\operatorname{KZ}_c$ and $\operatorname{Sh}$.
\end{Prop}
\begin{proof}
Since both $\operatorname{KZ}_c, \operatorname{Sh}$ are fully faithful on the projectives,
it is enough to check that $\operatorname{KZ}_c(P_c(\lambda))\cong Y(\lambda)$.
Let us define the object $I\Delta_c(\lambda)\in \OCat_c(n)$
as follows:
$$I\Delta_c(\lambda)=
\,^{\mathcal{O}}\operatorname{Ind}^{\Sym_{\lambda}}_{\Sym_n}(\Delta_{|\lambda_1|}((\lambda_1))\boxtimes \Delta_{|\lambda_2|}((\lambda_2))\boxtimes\ldots\boxtimes \Delta_{|\lambda_k|}((\lambda_k))).$$

Note that the object $I\Delta_c(\lambda)$ is projective in $\OCat_c(n)$. Indeed,
each $\Delta_c((\lambda_i))$ is projective in $\OCat_c(|\lambda_i|)$ because the label
$(\lambda_i)$ is maximal in the highest weight order. Since the induction functor
maps projective objects to projective objects, we see that $I\Delta_c(\lambda)$
is projective.  Further, from Lemma \ref{Lem:IndRes_stand} we deduce that
$\lambda$ is the minimal label of a standard that appears in the filtration of
$I\Delta_c(\lambda)$ and it occurs once.
From here we deduce that $I\Delta_c(\lambda)\cong P_c(\lambda)\oplus \bigoplus_{\mu>\lambda}
P_c(\mu)^{\oplus ?}$. By \cite[6.2.1]{GGOR}, $\operatorname{KZ}_c(\Delta_c((n))\cong \operatorname{Triv}_n$.
By Corollary \ref{Cor:KZ_Ind}, $$\operatorname{KZ}_c(I\Delta_c(\lambda))\cong IT(\lambda).$$
We conclude that $\operatorname{KZ}_c(\Delta_c(\lambda))=Y(\lambda)$.
\end{proof}

We finish this section with the following lemma proved in \cite[Section B.3]{SV}.

\begin{Lem}
The equivalence $\varphi$ intertwines $\Ind_{\Sym_{n}}^{\Sym_{n_1}\otimes \Sym_{n_2}}$
with $\bullet\otimes \bullet$.
\end{Lem}


\subsubsection{Modules $L_c(e\mu)$}
Suppose that $n$ is divisible by $e$, the denominator of $c$.
The quantum Frobenius homomorphism $U_\epsilon(\gl_N)\twoheadrightarrow U_{(-1)^e}(\gl_N)$
(see \cite[35.1]{Lusztig}), gives rise to the epimorphism $S_\epsilon(N,n)\twoheadrightarrow S_{(-1)^e}(N,n/e)$.
This gives rise to an exact functor $S_{-1}(N,n/e)\operatorname{-mod}\rightarrow
S_\epsilon(N,n)\operatorname{-mod}$ to be denoted by $\operatorname{Fr}^*$.
The objects in the essential image of $\operatorname{Fr}^*$ are the
sums of the simple objects $L^S(e\mu)$.

Thanks to Proposition \ref{Prop:equivalence_typeA}, we can determine the class of $L_c(e\mu)$
in $K_0(\OCat_c(n))$. Recall that the latter is identified with
$K_0(\Sym_n\operatorname{-mod})$ that in its turn is identified
with the space of degree $n$ symmetric polynomials in $\Z[z_1,z_2,\ldots]$ so that the class $[\lambda]$
becomes the Schur polynomial $s_\lambda$.

The following result is implicit in \cite[Proposition 5.13]{SV} and is more explicit in
\cite[Theorem 1.4]{EGL}.

\begin{Lem}\label{Lem:min_supp_class}
The class of $[L_c(e\mu)]$ in $K_0(\OCat_c(n))$ corresponds to
$s_\mu(z_1^e,z_2^e,\ldots)$.
\end{Lem}

\subsection{Categorical Heisenberg action}
Let $\kappa<0$ be a rational number with denominator $e$. Suppose that $\kappa e s_0,\ldots,\kappa e s_{\ell-1}$
be integers. These data define an  action of $\hat{\slf}_e$ on the level $\ell$
Fock space $\mathcal{F}^\ell$. Recall that we write $\mathcal{F}_{\kappa,\bf s}$ for $\mathcal{F}^\ell$
equipped with this action.

\subsubsection{Heisenberg action on the Fock space}
There is an action of the Heisenberg algebra $\mathfrak{Heis}$ on $\mathcal{F}_{\kappa,\bf s}$
commuting with the action of $\hat{\slf}_e$. Recall that $\mathfrak{Heis}$ has generators
$b_i, i\in \Z$, where $[b_i,b_j]=i\delta_{i,j}b_0$. We are only interested in the
representations, where $b_0$ acts by $1$.

So let us start by defining a Heisenberg action on the level one Fock $\mathcal{F}_{\kappa,a}, a\in \Z$
(it will be independent of the choice of $a$).
Let $i>0$. Define an operator $b_i$ on $\mathcal{F}_{\kappa,a}=\bigoplus_{n\geqslant 0}
K_0^{\C}(\Sym_n)$ by
$$[V_\lambda]\mapsto \sum_{j=0}^{ie-1}[\Ind^{\Sym_n\times \Sym_{ie}}_{\Sym_{n+ie}}V_{\lambda}\boxtimes
V_{(ie-j,1^{j})}].$$
Define the operator $b_{-i}$ as the adjoint of $b_i$ with respect to the symmetric bilinear
form on $\mathcal{F}_{\kappa,a}$ given by $(|\lambda\rangle, |\mu\rangle)=\delta_{\lambda\mu}$.
One can check that this indeed defines a representation of $\mathfrak{Heis}$ in
$\mathcal{F}_{\kappa,a}$.

In the realization of  $\sum_{a\in \Z}\mathcal{F}_{\kappa,a}$ as
the semi-infinite wedge $\Lambda^{\infty/2}\C^{e}[z^{\pm 1}]$, the operator
$b_i, i>0,$ is given by $$v_{i_1}\wedge v_{i_2}\wedge\ldots\mapsto
\sum_{j=1}^\infty v_{i_1}\wedge\ldots \wedge v_{i_{j-1}}\wedge
v_{i_j+ei}\wedge v_{i_{j+1}}\wedge\ldots$$
This realization shows that the actions of $\mathfrak{Heis}$
and $\hat{\slf}_e$ commute.

We take the tensor product representation of $\mathfrak{Heis}$ in $\mathcal{F}_{\kappa,\bf s}=\bigotimes_{i=0}^{\ell-1}
\mathcal{F}_{\kappa,s_i}$. It commutes with the representation of $\hat{\slf}_e$.

\subsubsection{Functors $A_\mu$}
Let $\mu$ be a partition of some positive integer $d$.
Following \cite[Section 5.1]{SV}, we consider  functors $A_\mu: \OCat_\bc\rightarrow \OCat_\bc$
mapping $\OCat_\bc(n)$ to $\OCat_\bc(n+de)$. These functors are defined by
$$A_\mu(M):=\Ind^{G(\ell,1,n)\times \Sym_{de}}_{G(\ell,1,n+de)}M\boxtimes L_c(e\mu).$$
These functors have derived right adjoint functors
$$RA_\mu^*(N):=R\Hom_{\OCat_{c}(de)}(L_c(e\mu),\Res_{G(\ell,1,n)\times \Sym_{de}}^{G(\ell,1,n+de)}).$$

The following results were established in \cite{SV}. The first part follows
from Proposition 5.4 there and the second part is Proposition 5.15.

\begin{Prop}\label{Prop:Heis_fun}
The following is true.
\begin{itemize}
\item $A_\mu\circ A_{\mu'}\cong A_{\mu'}\circ A_\mu$ for any partitions $\mu,\mu'$.
\item $A_\mu\circ E_i\cong E_i\circ A_\mu, A_\mu\circ F_i=F_i\circ A_\mu$.
\end{itemize}
\end{Prop}

By adjointness, we get the following claim.

\begin{Cor}\label{Cor:Heis_fun}
We have  $RA_{\mu}^*\circ RA_{\mu'}^*\cong RA^*_{\mu'}\circ RA_\mu^*$ and $RA_\mu^*\circ E_i\cong E_i\circ RA_\mu^*, RA^*_\mu\circ F_i\cong F_i\circ RA_\mu^*$.
\end{Cor}

\subsubsection{From functors $A_\mu$ to a Heisenberg action}
Let $h_d$ denote the power symmetric polynomial of degree $d$, i.e.,
$h_d=\sum_{i=1}^\infty z_i^d$. We have the following identity for symmetric functions:
\begin{equation}\label{eq:symm_fun_id} h_d=\sum_{i=0}^{d-1}(-1)^i s_{(d-i,1^{i})}.\end{equation}
From Lemma \ref{Lem:min_supp_class} combined with (\ref{eq:symm_fun_id}) it follows that
\begin{equation}\label{eq:K_0_equality}\sum_{i=0}^{ed-1} (-1)^i [\Delta_\bc((d-i,1^i))]=\sum_{i=0}^{d-1}(-1)^i[L_\bc((d-i)e,e^{i-1})].
\end{equation}

\begin{Cor}\label{Cor:Heis_oper}
The operator $\sum_{i=0}^{d-1} (-1)^i[A_{(d-i,1^i)}]$ on $K_0(\OCat_\bc)=\mathcal{F}_{\kappa,\bf s}$
coincides with the Heisenberg operator $b_d$. The operator
$\sum_{i=0}^{d-1} (-1)^i[RA^*_{(d-i,1^i)}]$ coincides with $b_{-d}$.
\end{Cor}
\begin{proof}
The claim about the operators $b_d$ follows from (\ref{eq:K_0_equality}) and the
definition of $b_d$. To prove the claim about the operators $b_{-d}$ we observe
that the symmetric form on $\mathcal{F}_{\kappa,\bf s}$ is the Euler form on $K_0$:
$([M],[N])=\sum_{i\geqslant 0}(-1)^i \dim \Ext^i(M,N)$. To see the coincidence of
the forms, we recall that $[\nabla_\bc(\mu)]=[\Delta_\bc(\mu)]$ and $\dim\Ext^i(\Delta_\bc(\lambda),
\nabla_\bc(\mu))=\delta_{\lambda \mu}\delta_{i0}$.
\end{proof}

\subsubsection{Application to finite dimensional modules in $\OCat_\bc(n)$}
The Heisenberg categorical action (meaning the functors $A_\mu, RA^*_\nu$)
was used in \cite{SV} to determine the filtration by supports
on $K_0(\OCat_\bc(n))$. We will discuss this in more detail in
Section \ref{S_supp}. For now we will just explain the first step in
this direction.

\begin{Prop}\label{Prop:fin_dim_class}
Let $L\in \Irr(\OCat_\bc(n))$ be finite dimensional. Then $e_i[L]=b_{-d}[L]=0$.
\end{Prop}
\begin{proof}
The condition that $L$ is finite dimensional is equivalent to $\Res^W_{\underline{W}}L=0$
for any proper parabolic subgroup $\underline{W}\subset W:=G(\ell,1,n)$.
In particular, we see that $EL=0$ and $RA_\mu^*L=0$ for any $\mu$.
The former equality implies $E_iL=0$ for any $i$, while the latter
yields $b_d [L]=0$ for any $d$, thanks to Corollary \ref{Cor:Heis_oper}.
\end{proof}

\section{Supports of simple modules}\label{S_supp}
In this section we study the supports of simple modules in the category $\mathcal{O}_{\kappa,{\bf s}}$.
We determine all theoretically possible supports in Section \ref{SS_pos_support}. Then we relate the supports
to crystal structures on $\mathcal{P}_\ell$. Then we introduce and compute the filtration by supports
on $K_0^{\C}(\OCat_{\kappa,{\bf s}})$ following \cite{SV}. Finally, we explain combinatorial recipes to
compute the support of $L_{\kappa,{\bf s}}(\lambda)$ following \cite{Cher_supp}.

\subsection{Possible supports}\label{SS_pos_support}
Let $e$ denote the denominator of $\kappa$. Set $W=G(\ell,1,n)$. For non-negative integers
$p,q$ satisfying $p+eq\leqslant n$ (if $\kappa$ is irrational, we assume that $e=+\infty$,
and so $q$ is automatically $0$), we set $W_{p,q}=G(n-p-eq,1,\ell)\times \Sym_e^{q}$,
this is the stabilizer of the point of the form $(x_1,\ldots,x_p,y_{1},\ldots,y_1,\ldots,
y_q,\ldots,y_q,0,\ldots,0)$, where $x_1,\ldots,x_p,y_1,\ldots,y_q$ are pairwise different
complex numbers and each $y_1,\ldots,y_q$ occurs $e$ times.

When  $e=1$, we assume that $p=0$.

Recall the subvariety $\overline{X(W_{p,q})}:=W\h^{W_{p,q}}\subset \h$.

\begin{Lem}\label{Lem:poss_supports}
Let $L\in \Irr(\OCat_{\kappa,{\bf s}}(n))$. Then there are $p,q$ with $p+eq\leqslant n$
such that $\Supp(L)=\overline{X(W_{p,q})}$. If $\ell=1$, then $p=0$.
\end{Lem}
\begin{proof}
We know, Lemma \ref{Lem:supp_simple}, that $\Supp(L)=\overline{X(\underline{W})}$ for some parabolic subgroup
$\underline{W}\subset W$. Also we know that $\Res^W_{\underline{W}}(L)$ is nonzero
finite dimensional. By \cite{BEG2},  there is a finite dimensional representation
in $\mathcal{O}_c(m)$ if and only if $m$ is the denominator of $c$.
The subgroup $\underline{W}$ is conjugate to $G(k,1,\ell)\times \Sym_{m_1}\times\ldots\Sym_{m_s}$
for some $k,m_1,\ldots,m_s$ (when $\ell=1$ we assume that $k=0$). We must have
$m_1=\ldots=m_s=e$. This implies the claim of the lemma.
\end{proof}

\subsection{Supports vs crystals}\label{SS_supp_cryst}
For $\lambda\in \mathcal{P}_\ell(n)$, we write $p_{\kappa,{\bf s}}(\lambda), q_{\kappa,{\bf s}}(\lambda)$
for the numbers $p,q$ such that $\Supp(L_{\kappa,{\bf s}}(\lambda))=\overline{X}(W_{p,q})$.

\subsubsection{$p_{\kappa,{\bf s}}(\lambda)$ and $\g_{\kappa,{\bf s}}$-crystal}\label{SSS_KM_crystal}
Let $\g_I$ be a type A Kac-Moody algebra and let $C$ be a highest weight $\g_I$-crystal,
where ``highest weight'' means that for every $v\in C$ there is $k\in \Z_{\geqslant 0}$
such that $\tilde{e}_{i_1}\ldots\tilde{e}_{i_k}v=0$ for any $i_1,\ldots,i_k\in I$.
We define the {\it depth} of $v$ as $k-1$ for the minimal such $k$. For example, $v$
has depth $0$ if and only if $\tilde{e}_i v=0$ for all $i\in I$.

The $\g_{\kappa,{\bf s}}$-crystal $\mathcal{P}_{\kappa,{\bf s}}$ is highest weight.
The following result was obtained in \cite[Section 5.5]{cryst}.

\begin{Prop}\label{Prop:cryst_depth_p}
The number $p_{\kappa,{\bf s}}(\lambda)$ coincides with the depth of $\lambda$
in $\mathcal{P}_{\kappa,{\bf s}}$.
\end{Prop}

\subsubsection{The map $\tilde{a}_\mu$}
Now we start to explain how to compute $q_{\kappa,{\bf s}}$.
Assume for the time being that $\kappa$ is a negative rational number with denominator $e$
and $\kappa e s_1,\ldots, \kappa e s_\ell$ are integers. The general case can be reduced
to this one, we will explain how to do this below.

Let $\mathcal{P}^{sing}_{\kappa,{\bf s}}$ denote the set of depth $0$
elements in the $\g_{\kappa,{\bf s}}$-crystal $\mathcal{P}_{\kappa,{\bf s}}$.
We write $\mathcal{P}^{sing,q}_{\kappa,{\bf s}}$ for the subset of all $\lambda\in
\mathcal{P}^{sing,q}_{\kappa,{\bf s}}$ such that $q_{\kappa,{\bf s}}(\lambda)=q$.

We are going to define a map $(\mu,\lambda)\mapsto \tilde{a}_\mu \lambda$
from $\mathcal{P}_1(q)\times \mathcal{P}^{sing,0}_{\kappa,{\bf s}}$
to $\mathcal{P}^{sing,q}_{\kappa,{\bf s}}$ following \cite[Section 5]{SV} that happens
to be a bijection.

Pick $\lambda\in \mathcal{P}^{sing,0}_{\kappa,{\bf s}}(n)$ (i.e., $\lambda\in \mathcal{P}_\ell(n)$
such that $L_{\kappa,{\bf s}}(\lambda)$ is finite dimensional) and consider the object
$$A_\mu L_{\kappa,{\bf s}}(\lambda)=\Ind^{G(n,1,\ell)\times \Sym_{e|\mu|}}_{G(n+e|\mu|,1,\ell)}L_{\kappa,{\bf s}}(\lambda)\boxtimes L_{\kappa}(e\mu).$$

The following is a key result about the structure of $A_\mu L_{\kappa,{\bf s}}(\lambda)$,
see \cite[Sections 5.4-5.6]{SV}.

\begin{Prop}\label{Prop:Amu_structure}
The following is true:
\begin{enumerate}
\item  The head and the socle of $A_\mu L_{\kappa,{\bf s}}(\lambda)$ are isomorphic
simple objects with label to be denoted by $\tilde{a}_\mu\lambda$.
We have $\tilde{a}_\mu\lambda\in \mathcal{P}^{sing,|\mu|}_{\kappa,{\bf s}}$.
\item The multiplicity of $L_{\kappa,{\bf s}}(\tilde{a}_\mu\lambda)$
in $A_\mu L_{\kappa,{\bf s}}(\lambda)$ equals $\ell^{|\mu|}\dim V_{\mu}$,
where we write $V_{\mu}$ for the irreducible $\Sym_{|\mu|}$-module
labelled by $\mu$.
\item Let $\lambda'\neq\tilde{a}_\mu\lambda$ be such that $L_{\kappa,{\bf s}}(\lambda')$
is a simple constituent of $A_\mu L_{\kappa,{\bf s}}(\lambda)$. Then
$p_{\kappa,{\bf s}}(\lambda')=0$ and $q_{\kappa,{\bf s}}(\lambda')<|\mu|$.
\end{enumerate}
\end{Prop}

So we indeed get a map between the required sets. The following is \cite[Proposition 5.33]{SV}.

\begin{Prop}\label{Prop:amu_bijection}
The map $(\mu,\lambda)\mapsto \tilde{a}_\mu \lambda$ is a bijection
$$\mathcal{P}_1(q)\times \mathcal{P}^{sing,0}_{\kappa,{\bf s}}
\xrightarrow{\sim}\mathcal{P}^{sing,q}_{\kappa,{\bf s}}.$$
\end{Prop}

\subsubsection{Heisenberg crystal}\label{SSS_Heis_cryst}
The set $\mathcal{P}_1$ comes with a natural $\slf_\infty$-crystal structure
that is the same as on $\mathcal{P}_\kappa$ for irrational $\kappa$.
The bijection $\mathcal{P}_1\times \mathcal{P}^{sing,0}_{\kappa,{\bf s}}\xrightarrow{\sim}
\mathcal{P}^{sing}_{\kappa,{\bf s}}$ allows to carry this structure over to
$\mathcal{P}^{sing}_{\kappa,{\bf s}}$.  Then we have the following result
proved in \cite[5.1.2]{Cher_supp}.

\begin{Prop}\label{Prop:slf_infty_crystal}
There is a unique $\slf_\infty$-crystal structure on $\mathcal{P}_{\kappa,{\bf s}}$ that commutes
with the $\hat{\slf}_e$-crystal structure and extends the $\slf_\infty$-crystal structure
on $\mathcal{P}_{\kappa,{\bf s}}^{sing}$ introduced in the previous paragraph.
The number $q_{\kappa,{\bf s}}(\lambda)$ coincides with the depth of $\lambda$
in this crystal.
\end{Prop}

Since this $\slf_\infty$-crystal is a crystal analog of a Heisenberg algebra action,
we call it the {\it Heisenberg crystal}.

Comparing to the $\hat{\slf}_e$-crystal operators, the crystal operators for the $\hat{\slf}_\infty$-crystal
are hard to compute. A relatively easy case is when one of the parameters $s_i$ is much less than
the others. In order to compute the crystal operators in this case,
let us recall the division with remainder for Young diagrams.
For a Young diagram $\nu=(\nu_1,\ldots,\nu_k)$ and an integer $e>1$ we get the partial quotient
$\nu'=(\nu'_1,\ldots,\nu'_k)$ and the remainder $\nu''=(\nu''_1,\ldots,\nu''_k)$, both are Young diagrams. The rule
to determine $\nu',\nu''$ is that $\nu_i=e\nu'_i+\nu_i''$ for all $i$ and that $|\nu''|$ is maximal
possible. For example, if $\nu=(7,3,1)$ and $e=3$, then we get $\nu'=(1),\nu''=(4,3,1)$.
Note  that the partition $\nu$ lies in the $\hat{\slf}_e$-crystal component of
$e \nu'$ and goes to $\nu''$ under the isomorphism of the components of $\varnothing$
and of $e\nu'$, see Example \ref{Ex:level1_cryst}.

\begin{Prop}\label{Prop:Heis_crystal_domin}
Assume that $j\in \{0,1,\ldots,\ell-1\}$ is such that $s_j<s_i-n$ for any other $i$.
Let $\lambda\in \mathcal{P}_{\kappa,{\bf s}}(n)$. Let $\lambda^{(j)}=e\nu'+\nu''$
(division with remainder). Then $q_{\kappa,{\bf s}}(\lambda)=|\nu'|$. Moreover,
the crystal operator $\tilde{e}_j, j\in \Z,$ for $\slf_\infty$ sends
$\lambda$ to
$$\underline{\lambda}:=(\lambda^{(1)},\ldots,\lambda^{(j-1)}, e(\tilde{e}_j\nu')+\nu'',\lambda^{(j+1)},
\ldots, \lambda^{(\ell)}),$$
if $\tilde{e}_j\nu'\neq 0$ and to $0$ else. Here $\tilde{e}_j\nu'$ is computed in the
$\slf_\infty$-crystal on the set of partitions.
\end{Prop}

This is proved in \cite[Section 5.2]{Cher_supp}. Below, we will explain how to compute the crystal
operators for $\slf_\infty$, in general, using the wall-crossing bijections.

\subsubsection{Reduction}
Let us explain how to reduce the computation of $q_{\kappa,{\bf s}}$ to the case
when $\kappa<0$ and $\kappa e s_i\in \Z$ for all $i$.

First of all, assume that $\kappa>0$. Let $\chi$ be the one-dimensional character
of $G(\ell,1,n)$ that is the identity of $\mu_\ell$ and is the sign on $\Sym_n$.
Then we have an equivalence $\OCat_{\kappa,{\bf s}}\xrightarrow{\sim} \OCat_{-\kappa,-{\bf s}}$
that sends $\Delta(\lambda)$ to $\Delta(\lambda^t)$ (and hence $L(\lambda)$
to $L(\lambda^t)$), where $\bullet^t$ means the component wise transpose.
This reduces the computation of $q_{\kappa,{\bf s}}$ to the case when $\kappa<0$.

Now let us explain what to do if not all $\kappa e s_0,\ldots \kappa e s_{\ell-1}$
are integers (up to adding the same summand to $s_0,\ldots,s_{\ell-1}$). This is
precisely the case when there is more than one equivalence class for $\sim_\bc$
in $\{0,1,\ldots,\ell-1\}$ (see \ref{SSS_decomp1} for the definition of $\sim_\bc$).
The following result follows easily from Proposition \ref{Prop:cat_O_decomp},
see \cite[Corollary 4.3]{Cher_supp}.

\begin{Lem}\label{Lem:q_decomp}
We have $q_{\kappa,{\bf s}}(\lambda)=\sum_{\alpha} q_{\kappa,{\bf s}^\alpha}(\lambda^\alpha)$.
\end{Lem}

\subsection{Filtration  by support on $K_0$}\label{SS_filtr_K0}
Propositions \ref{Prop:Amu_structure}, \ref{Prop:amu_bijection} together with standard properties of
$\hat{\slf}_e$-crystals were used in \cite{SV} to describe the filtration
by supports on $K_0^{\C}(\OCat_{\kappa,{\bf s}})=\mathcal{F}_{\kappa,{\bf s}}$.
Namely, for $p,q\in \Z_{\geqslant 0}$, let $\OCat_{\kappa,{\bf s}}^{p,q}$ denote
the span of $[L_{\kappa,{\bf s}}(\lambda)]$ with $p_{\kappa,{\bf s}}(\lambda)\leqslant p,
q_{\kappa,{\bf s}}(\lambda)\leqslant q$. Further, we let $\mathcal{F}_{\kappa,{\bf s}}^{p,q}$
to denote the span of elements of the form $$f_{i_1}\ldots f_{i_{p}}b_{j_1}\ldots b_{j_k}v,$$
where $p'\leqslant p, j_1+\ldots+j_k\leqslant q$ and $v\in \mathcal{F}_{\kappa,{\bf s}}$
is a vector annihilated by all  $e_i$ and all $b_{-j}$.

The following result was obtained in \cite{SV}, see the proof of Proposition 6.5 there.

\begin{Prop}\label{Prop:K0_filtr}
We have $K_0^{\C}(\mathcal{O}^{p,q}_{\kappa,{\bf s}})=\mathcal{F}^{p,q}_{\kappa,{\bf s}}$.
\end{Prop}

\subsection{Combinatorial computation of supports via wall-crossing bijections}\label{SS_supp_comb_comput}
We assume that $\kappa$ is a negative rational number with denominator $e$
and the numbers $\kappa e s_0,\ldots, \kappa e s_{\ell-1}$ are integers.

Thanks to Proposition \ref{Prop:Heis_crystal_domin} we can compute the numbers
$q_{\kappa,{\bf s}}(\lambda)$ in the case when one of the numbers $s_j$
is much less than the others. We call (essential) chambers, see \ref{SSS_essent_walls_cyclot},
satisfying these conditions {\it asymptotic}. In order to compute $q_{\kappa,{\bf s}}(\lambda)$
(and the crystal operators for the Heisenberg crystal) we will ``move''
a general parameter $(\kappa,{\bf s})$ to an asymptotic chamber by applying
wall-crossing bijections through the essential walls $s_i-s_j=m$, where $m$ is an integer
(with fixed residue modulo $e$). We will fix $j$ and  use different $i$.

Two crucial properties that allow us to compute this wall-crossing bijection
are as follows.

\begin{Prop}
Let $\bc=(\kappa, {\bf s})$ and $\bc'=(\kappa', {\bf s}')$ be two parameters with
integral difference whose chambers are separated by the wall $\Pi$ given by
$h_i-h_j=\kappa m$ for $m\in \Z$. Let $\wc_{\bc'\leftarrow \bc}:\mathcal{P}_{\ell}
\rightarrow \mathcal{P}_\ell$. Then the following is true.
\begin{enumerate}
\item For a Zariski generic $\tilde{\bc}\in \bc+\Pi, \tilde{\bc}':=\tilde{\bc}+(\bc'-\bc)$,
the bijection $\wc_{\tilde{\bc}'\leftarrow \tilde{\bc}}$ is independent of the choice of $\tilde{\bc}$.
\item The bijection $\wc_{\tilde{\bc}'\leftarrow \tilde{\bc}}$ commutes with
the $\g_{\tilde{\bc}}$-crystal.
\end{enumerate}
\end{Prop}
The first part is a special case of Proposition \ref{Prop:wc_bij_indep}.
The second part is an easy consequence of \ref{Prop:WC_commut_IndRes}.

The bijection $\wc_{\bc'\leftarrow \bc}$ is then computed as follows. We choose a Weil generic
parameter $\tilde{\bc}$. In this case there are $\ell-1$ equivalence classes for
$\sim_{\tilde{\bc}}$, and $i\sim_{\tilde{\bc}}j$. Furthermore,   $\g_{\tilde{\bc}}=\slf_\infty^{\ell-1}$
and the crystal structure on $\mathcal{P}_{\tilde{\bc}}$ is the product of $\ell-2$
of the level $1$ crystal for $\slf_\infty$ and one copy of a level $2$ crystal
for $\slf_\infty$, the latter affects partitions number $i$ and $j$. In fact,
there is a unique isomorphism  of crystals  $\mathcal{P}_{\tilde{\bc}}\rightarrow
\mathcal{P}_{\tilde{\bc'}}$ that maps $\mathcal{P}_{\tilde{\bc}}(n)$ to $\mathcal{P}_{\tilde{\bc}'}(n)$,
this can be deduced from Examples \ref{Ex:level2_cryst_1},\ref{Ex:level2_cryst_2}.
This uniqueness allows to recover $\wc_{\bc'\leftarrow \bc}$. For details the reader is referred
to \cite[Section 5.4]{Cher_supp}.

\section{Category equivalences and multiplicities}\label{S_cat_equi}
\subsection{Kazhdan-Lusztig category}\label{SS_KL}
Recall that we have a highest weight category equivalence $\OCat_{c}(n)\xrightarrow{\sim}
\mathcal{S}_\epsilon(m,n)\operatorname{-mod}$, where $\mathcal{S}_{\epsilon}(m,n)\operatorname{-mod}$
is the Serre subcategory in $U_\epsilon(\gl_m)\operatorname{-mod}_{fin}$ generated by simples
$L(\lambda)$, where $\lambda$ is a highest weight of the form $(\lambda_1,\ldots,\lambda_m)$,
where $\lambda_1\geqslant \lambda_2\geqslant\ldots\geqslant \lambda_m\geqslant 0$.

The category $U_\epsilon(\gl_m)$ has an alternative realization, the so called {\it Kazhdan-Lusztig
category} of representations of the affine Lie algebra $\hat{\g}:=\hat{\gl}_m$.
Namely, let $\hat{G}$ denote the corresponding Kac-Moody group, a central
extension of $\operatorname{GL}_m(\C((t)))$, $I\subset \hat{G}$
be the Iwahori subgroup. We consider the category $\mathcal{O}^{\mathfrak{b}}_{\kappa}$
consisting of all $I$-integrable $\hat{\g}$-modules with weights bounded from above
and such that the standard central element $C\in \hat{\g}$ acts by $\kappa^{-1}-m$
(here we assume that $\kappa$ is negative).
Set $G_+:=\operatorname{GL}_m(\C[[t]])\subset \hat{G}$. By the Kazhdan-Lusztig
category (to be denoted by $\mathcal{O}^{\g}_{\kappa}$)
one means the full subcategory of $\mathcal{O}^{\mathfrak{b}}_\kappa$
consisting of $G_+$-integrable modules. This is a highest weight category
whose standard objects are the so called Weyl modules $$W(\lambda):=U_{\kappa^{-1}-m}(\hat{\g})\otimes_{U(\g_+)}V(\lambda),$$
where we write $V(\lambda)$  for the finite dimensional irreducible $\g$-module with highest weight $\lambda$.

A deep result of Kazhdan and Lusztig (valid for any reductive group)
is that there is a highest weight category equivalence $\OCat^{\g}_\kappa\cong
U_\epsilon(\g)\operatorname{-mod}_{fin}$. In order to produce this equivalence,
Kazhdan and Lusztig introduce a braided monoidal structure on $\OCat^{\g}_\kappa$,
the so called {\it fusion product} $\dot{\otimes}: \OCat^{\g}_{\kappa}\times
\OCat^{\g}_{\kappa}\rightarrow \OCat^{\g}_{\kappa}$. In order to compute
the fusion product of modules $V_1,\ldots,V_k\in \OCat^{\g}_{\kappa}$
one needs to fix different points $z_1,\ldots,z_k$; the different choices
of points lead to isomorphic products $V_1\dot{\otimes}V_2\dot{\otimes}\ldots
\dot{\otimes}V_k$ that form an $\Sym_k$-equivariant  local system on $\{(z_1,\ldots,z_k)\in \C^k| z_i\neq z_j\}$
(the $\Sym_k$ action permutes both coordinates and the tensor factors).
The monodromy of this system gives the braided structure.

\subsection{Truncated affine parabolic category $\mathcal{O}$}
Now let $m\geqslant n$ and let $\g=\mathfrak{gl}_m$. Let $\OCat^{\g}_\kappa(n)$
denote the Serre subcategory of $\OCat^\g_\kappa$ spanned by the Weyl modules
$W(\lambda)$ with $\lambda_m\geqslant 0, \sum_{i=1}^m \lambda_i=n$.
We have a highest weight equivalence $\mathcal{S}_\epsilon(m,n)\operatorname{-mod}
\xrightarrow{\sim} \OCat^{\g}_\kappa(n)$ and hence a highest weight
equivalence $\OCat_{\kappa}(n)\xrightarrow{\sim} \OCat^{\g}_\kappa(n)$.
The point of emphasizing the latter equivalence is that it can be
generalized to $\ell>1$ as was conjectured by Varagnolo-Vasserot in
\cite{VV} and proved in \cite{RSVV,VV_proof}.

Let us explain the conjecture by introducing  ``higher level'' analogs of
the categories $\OCat^{\g}_\kappa(n)\subset \OCat^{\g}_\kappa$.
We assume that $\kappa=-\frac{1}{e}, s_0,\ldots,s_{\ell-1}$
are integers, in the end of the section we will explain what to
do in the general case.

We   modify $s_0,\ldots,s_{\ell-1}$ by adding a common summand
so that $s_i>n$ for all $i$. Set $m:=s_0+\ldots+s_{\ell-1}$. Let
$\p\subset \g$ be the parabolic subalgebra of block upper triangular
matrices, where the blocks have sizes $s_0,s_1,\ldots,s_{\ell-1}$
(in this order). We have the corresponding parahoric subgroup $\hat{P}
\subset \hat{G}$. Consider the full subcategory $\OCat^{\p}_{\kappa}\subset
\OCat^{\g}_{\kappa}$ consisting of all $\hat{P}$-integrable modules.
This is ``almost'' a highest weight category in the following sense.

Consider the set $\Z^{\bf s}=\{(a_1,\ldots,a_m)| a_1>\ldots>a_{s_0}, a_{s_0+1}>\ldots>a_{s_0+s_1},\ldots,
a_{s_0+\ldots+s_{\ell-2}+1}>\ldots>a_m\}$. For $A\in \Z^{\bf s}$, let $\Delta^{\p}(A)$ denote the parabolic
Verma module with $\rho$-shifted highest weight $A$ so that $\OCat^{\p}_\kappa$
is the Serre span of $\Delta^{\p}(A)$. There is a so called {\it linkage order} on $\Z^{\bf s}$
(that we are not going to define). It has the property that for all $A\in \Z^{\bf s}$ the poset ideals
$\{A'\in \Z^{\bf s}| A'\leqslant A\}$ are finite. It turns out that, for any finite poset ideal
$\Lambda\subset \Z^{\bf s}$, the Serre span of $\Delta^{\p}(A), A\in \Lambda,$ is a highest
weight category with standard objects $\Delta^{\p}(A)$. The category $\mathcal{O}^{\p}_{\kappa}$
does not have enough projective objects (but has enough projective pro-objects) and so
it is not highest weight in the sense introduced in \ref{SSS_HW}.

Now recall the poset $\mathcal{P}_\ell(n)$ with order $\preceq_\bc$, see \ref{SSS_O_order}.
We can embed $\mathcal{P}_\ell(n)$ into $\Z^{\bf s}$
by $(\lambda^{(0)},\ldots,\lambda^{(\ell-1)})\mapsto A_\lambda$, where $A_\lambda$ is given by
$$(s_0+\lambda^{(0)}_1,s_0-1+\lambda^{(0)}_2,s_0-2+\lambda^{(0)}_3,\ldots,1, s_1+\lambda^{(1)}_1,s_1+\lambda^{(1)}_2-1,
\ldots,1,\ldots, s_{\ell-1}+\lambda^{(\ell-1)}_1,\ldots,1).$$
Here we use the inequalities $n\leqslant s_i$ for all $i$.

One can show that $\mathcal{P}_\ell(n)\subset \Z^{\bf s}$ is a poset ideal with respect to the linkage order
and, moreover, the order $\preceq_\bc$ on $\mathcal{P}_\ell(n)$ refines the restriction of the  linkage
order, see, e.g., \cite[Section 2.3]{VV_proof}.

It is convenient for us to view elements of $\Z^{\bf s}$ as some infinite to the left analogs
of Young diagrams. Namely, we view an element of $\Z^{\bf s}$ as a sequence of $\ell$ collections of boxes
on the plane, where $i$th collection is infinite to the left, consists of $s_i$ rows,
and the rightmost positions of boxes in each row (non-strictly) decrease from top to bottom.
The embedding $\mathcal{P}_\ell(n)\hookrightarrow \Z^{\bf s}$ takes an $\ell$-tuple of diagrams
and adjoins the columns with nonpositive numbers of height $s_i$ to $\lambda^{(i)}$.

For $\OCat^{\p}_\kappa(n)$ we take the Serre span of $\Delta^{\p}(A_\lambda),\lambda\in \mathcal{P}_\ell(n)$.

The following conjecture was made in \cite[8.8]{VV} (under unnecessary restrictions on parameters).

\begin{Conj}\label{Conj:VV}
There is a highest weight equivalence $\OCat_{\kappa,{\bf s}}(n)\xrightarrow{\sim} \OCat^{\p}_{\kappa}(n)$.
\end{Conj}

In particular, this conjecture allows to prove that the category $\OCat_{\kappa,{\bf s}}(n)$
is standard Koszul and describe its Koszul dual, \cite[Section 7.3]{RSVV}, see also
\cite[Section 6]{Webster}.


\subsection{Categorical action on affine parabolic category}
In order to prove Conjecture \ref{Conj:VV} we will use techniques similar to what was
used in Section \ref{SS_abelian_equi}. First of all, we need to produce a
quotient functor $\pi:\OCat^{\p}_{\kappa}(n)\twoheadrightarrow \mathcal{H}_{q,{\bf s}}(n)\operatorname{-mod}$.
This is one of the places in the proof of Conjecture \ref{Conj:VV}, where a categorical
$\hat{\slf}_e$-action is used.

\subsubsection{Cartan component functor}\label{SSS_Cartan_functor}
Let $\omega_i, i\in \Z/e\Z,$ denote the fundamental weight for $\hat{\slf}_e$.
Set $\omega_{\bf s}:=\sum_{i=0}^{\ell-1}\omega_{s_i}$. Now let $\Cat$ be an $\hat{\slf}_{e}$-categorification.
Assume that $\Cat_{\omega_{\bf s}}\cong \operatorname{Vect}$, while $\Cat_{\eta}\neq 0$ implies
$\eta\leqslant \omega_{\bf s}$. In particular, $[\Cat]$ is a highest weight integrable
representation with highest weight $\omega_{\bf s}$. It follows that the irreducible
$\hat{\slf}_e$-module $V(\omega_{\bf s})$ with highest weight $\omega_{\bf s}$
is a direct summand of $[\Cat]$ (we call it the {\it Cartan component}, by analogy with
tensor product modules). The following result of Rouquier, \cite[5.1.2]{Rouquier_2Kac} should
be thought as a categorical analog of the existence of the Cartan component.

\begin{Prop}\label{Prop:Rouquier_unique}
Let ${\bf 1}$ denote the indecomposable projective object in $\Cat_{\omega_{\bf s}}$.
Then $\operatorname{End}_{\Cat}(F^n{\bf 1})^{opp}\cong \mathcal{H}^{{\bf s}}_{q}(n)$.
In particular, the projective object $\bigoplus_{n\geqslant 0} F^n {\bf 1}$
defines a quotient functor $\Cat\twoheadrightarrow \bigoplus_{n\geqslant 0} \mathcal{H}^{{\bf s}}_{q}(n)$
that is a morphism of $\hat{\slf}_e$-categorifications.
\end{Prop}

\subsubsection{Categorical action on $\OCat^{\p}_\kappa$}
So we would like to realize $\OCat^{\p}_{\kappa}(n)$ as the degree $n$
component in a categorification $\Cat$ like in \ref{SSS_Cartan_functor}.
We will do something weaker to be explained in \ref{SSS_cat_truncn}. Now we
are going to produce a categorical $\hat{\slf}_e$-action on  the whole
category $\OCat^{\p}_\kappa$.

For $M\in \OCat^{\g}_\kappa, N\in \OCat^{\p}_\kappa$, we can still form
the fusion product $M\dot{\otimes}N$ and this will be an object
of $\OCat^{\p}_{\kappa}$. If $M$ is standardly filtered, then
the endo-functor $M\dot{\otimes}\bullet$ is exact. Moreover, there is
a contravariant duality $D$ on $(\OCat_\kappa^{\p})^{\Delta}$
that maps $W(\mu)$ to $W(\mu^*)$ (where $\mu^*$ denotes the highest weight
of $V(\mu)^*$), see \cite[Sections 2.6,2.9]{VV_proof}. The functor $DM\dot{\otimes}\bullet$
is biadjoint to $M\dot{\otimes}\bullet$. See \cite[Corollary 7.3]{VV} for the proofs.

Moreover, we have an $\Sym_n$-equivariant local system $M^{\dot{\otimes} n}\dot{\otimes}N$
on $(\C^\times\setminus \{0\})^n$ (we place the $n$ copies of $M$ to different nonzero
points $z_1,\ldots,z_n\in \C$ and $N$ to $0$). So the affine braid group $B_n^{aff}$
acts by automorphisms on the functor $M^{\dot{\otimes} n}\dot{\otimes}\bullet$. The action of
the subgroup $B_n\subset B^{aff}_n$ comes from the $B_n$-action on $M^{\dot{\otimes} n}$.

Now let us define the categorical $\hat{\slf}_e$-action on $\OCat^{\p}_\kappa$.
Let $\mu$ be the highest weight of $\C^m$. Set $F:=W(\mu)\dot{\otimes}\bullet,
E:=W(\mu^*)\dot{\otimes}\bullet$. We have already mentioned that these
functors are biadjoint so we get (i). The $B^{aff}_n$-action on $F^n$
factors through $\mathcal{H}^{aff}_q(n)$. This is because the $B_n$-action
on $W(\mu)^{\dot{\otimes n}}$ factors through $\mathcal{H}_q(n)$,
which follows from results of Kazhdan and Lusztig recalled
\ref{SS_KL}.  So we get (ii).

In order to establish (iii) and (iv), we need to understand the behavior
of $F,E$ on the standard objects $\Delta^{\p}(A)$. There is a general
result of Varagnolo and Vasserot, \cite[Proposition A2.6]{VV}, that
$W(\mu)\dot{\otimes}\Delta^{\p}(A)$ is standardly filtered, and the
multiplicity of $\Delta^{\p}(A')$ in this
filtration coincides with the multiplicity of $\Delta^{\p}_{fin}(A')$ in $V(\lambda)\otimes
\Delta^{\p}_{fin}(A)$, where we write $\Delta^{\p}_{fin}(A)$ for the parabolic Verma
module for $(\g,\p)$ with $\rho$-shifted highest weight $A$.
It follows that $F\Delta^{\p}(A)$ has a filtration with all
possible quotients $\Delta^{\p}(A+\epsilon_i)$, where $\epsilon_i$
is the $i$th coordinate vector, each quotient occurs once. A similar
result holds for $E\Delta^{\p}(A)$ (we need to use $A-\epsilon_i$
instead of $A+\epsilon_i$). In other words, for a virtual $\ell$-partition
$\lambda$, the object $F\Delta^{\p}(\lambda)$ is filtered with
$\Delta^{\p}(\lambda')$, where $\lambda'$ is obtained from $\lambda$
by adding a box. One can show that the action of $X$ on the subquotient
$\Delta^{\p}(\lambda')$ is by $q^d$, where $d$ is the content of the box
$\lambda'\setminus \lambda$.

In particular, we have the following result. Consider the level $0$ $\hat{\slf}_e$-module
$\C^{\Z}$ with basis $v_i, i\in \Z,$ and the action given by $f_i v_j=\delta_{i,j\,\operatorname{mod}\,e}v_{j+1},
e_i v_{j+1}=\delta_{i,j\,\operatorname{mod}\,e}v_j$.

\begin{Cor}\label{Cor:K0_affine_parab}
We have an isomorphism $K_0^{\C}(\OCat^{\p}_\kappa)\cong \bigwedge^{s_0}\C^\Z\otimes \bigwedge^{s_1}\C^\Z\otimes\ldots\otimes
\bigwedge^{s_{\ell-1}}\C^\Z$ of $\hat{\slf}_e$-modules that maps $[\Delta^{\p}(A)]$ to the monomial vector
$$v_{a_1}\wedge v_{a_2}\ldots \wedge v_{a_{s_0}}\otimes v_{a_{s_0}+1}\wedge\ldots\wedge v_{a_{s_0+s_1}}\otimes\ldots$$
\end{Cor}

In particular, (iii) follows.

Let us explain why (iv) is true. The category $\OCat^{\p}_\kappa$ splits into the direct sum of blocks.
If two parabolic Verma modules $\Delta^{\p}(A)$ and $\Delta^{\p}(A')$ lie in the same block,
then $A\mod e$ and $A'\mod e$ are $\Sym_m$-conjugate. On the other hand, $[\Delta^{\p}(A)],[\Delta^{\p}(A')]$
lie in the same weight space if and only if $A\mod e, A'\mod e$ are $\Sym_m$-conjugate.
So for the weight subcategories we can take suitable sums of blocks.

\subsubsection{Categorical truncation}\label{SSS_cat_truncn}
Consider the subcategory $\OCat^{\p}_{\kappa}(\leqslant n):=\bigoplus_{j=0}^n
\OCat^{\p}_{\kappa}(j)$, where, recall, we assume that $n<s_i$. The functors $F_i$
map $\OCat^{\p}_{\kappa}(j)$ to  $\OCat^{\p}_{\kappa}(j+1)$ for $j<n$. The
functors $E_i, i\neq 0,$ send $\OCat^{\p}_{\kappa}(j)$ to $\OCat^{\p}_{\kappa}(j-1)$.
This is because all  removable $i$-boxes in the diagram representing $A_\lambda$ are
the same as in $\lambda$. However, the diagram of $A_\lambda$ contains $\ell$ removable
zero boxes  that are not in $\lambda$, those are the boxes $(0,s_i,i)$. So $E_0$
does not map $\OCat^{\p}_{\kappa}(j)$ to $\OCat^{\p}_{\kappa}(j-1)$.

Fortunately, this can be fixed because the $\ell$-boxes $(0,s_i,i)$ are the minimal
removable $0$-boxes in the diagram of $A_\lambda$. It follows from results of
\cite[Section 5]{hw_str} that the functor $F_0:\OCat^{\p}_\kappa(j-1)\rightarrow
\OCat^{\p}_\kappa(j)$ admits a biadjoint functor to be denoted by $\underline{E}_0$.
The functors $F: \OCat_{\kappa}^{\p}(\leqslant n-1)\rightarrow
\OCat^{\p}_{\kappa}(\leqslant n), \underline{E}:=\underline{E}_0\oplus \bigoplus_{i\neq 0}E_i:
\OCat^{\p}_{\kappa}(\leqslant n)\rightarrow \OCat^{\p}_{\kappa}(\leqslant n-1)$
give $\OCat^{\p}_{\kappa}(\leqslant n)$ the structure of a {\it restricted}
categorical $\hat{\slf}_e$-action. This categorical action categorifies
the restriction of the $\hat{\slf}_e$-action to $\mathcal{F}^\ell(\leqslant n):=
\bigoplus_{i=0}^n \mathcal{F}^\ell(n)$.

Proposition \ref{Prop:Rouquier_unique} is still true for restricted categorifications.
In particular, we get a quotient functor $\OCat^{\p}_\kappa(n)\twoheadrightarrow
\mathcal{H}^{\bf s}_q(n)\operatorname{-mod}$.

\subsection{Deformations}
So now we have quotient functors $\pi^1:\OCat^\p_{\kappa}(n)\twoheadrightarrow \mathcal{H}_\kappa^{\bf s}(n)\operatorname{-mod}, \pi^2:\OCat_{\kappa,{\bf s}}(n)\twoheadrightarrow \mathcal{H}_{\kappa}^{\bf s}(n)\operatorname{-mod}$. Similarly to Section \ref{SS_abelian_equi}, in order to produce
a highest weight equivalence $\OCat^\p_{\kappa}(n)\xrightarrow{\sim} \OCat_{\kappa,{\bf s}}(n)$
we will need to consider deformations of $\OCat^\p_{\kappa}(n), \OCat_{\kappa,{\bf s}}(n),
\mathcal{H}^{\bf s}_{\kappa}(n)\operatorname{-mod}$.  Unlike in {\it loc.cit.}, it does not seem
that one-parameter deformations are sufficient. Following a key idea of \cite{RSVV}, we will use
two parameter deformations.

Deformations of $\OCat_{\kappa,{\bf s}}(n), \mathcal{H}^{\bf s}_{\kappa}(n)\operatorname{-mod}$
are basically the same as in \ref{SSS_deform_cher}, but instead of a generic line through $\bc$
we need to take a generic plane. Let us explain how to deform $\OCat^{\p}_{\kappa}(n)$ together
with the quotient functor $\pi^2:\OCat^{\p}_\kappa(n)\twoheadrightarrow \mathcal{H}_{\kappa}^{\bf s}(n)\operatorname{-mod}$.

Pick formal variables $x_0,\ldots,x_\ell$ with $x_1+\ldots+x_\ell=0$. Set
$\tilde{R}=\C[[x_0,\ldots,x_\ell]]/(x_1+\ldots+x_\ell)$. We will deform our categories over
$\tilde{R}$ (and then we can specialize to a suitable algebra of formal power series in two variables).
First, let us define the deformed category $\OCat^{\p}_{\kappa,R}$. It  consists
of all $\hat{\g}\otimes \tilde{R}$-modules $M$ subject to the following conditions:
\begin{itemize}
\item The action of $\hat{\g}\otimes \tilde{R}$ on $M$ is $\tilde{R}$-linear. Moreover,
$M$ is finitely generated over $U(\hat{\g})\otimes \tilde{R}$.
\item The level of $M$ is $(x_0-\frac{1}{e})^{-1}$.
\item The action of the parabolic subalgebra $\hat{\p}\subset \hat{\g}$ is locally finite,
meaning that every element of $M$ is contained in a finitely generated $R$-submodule
stable under $\hat{\p}$.
\item For any $i$, the element $\operatorname{id}_i\in \gl_{s_i}\subset \p$
acts on $M$ diagonalizably with eigenvalues in $\Z+x_i$. Moreover, after specializing
the elements $x_i$ to $0$, the action of $\hat{\p}$ on $M/(x_0,\ldots,x_\ell)M$
integrates to $\hat{P}$.
\end{itemize}
This category still has parabolic Verma modules $\Delta^{\p}(A)$ labelled by
elements of $\Z^{\bf s}$. So we can form the highest weight subcategory
$\OCat^\p_{\kappa,\tilde{R}}(n)\subset \OCat^\p_{\kappa,\tilde{R}}$ whose poset
is $\mathcal{P}_\ell(n)$. We note that this category is equivalent to
$A_{\tilde{R}}\operatorname{-mod}$ for some $\tilde{R}$-algebra $A_{\tilde{R}}$
that is a free $\tilde{R}$-module of finite rank. So we can consider
the base change $\OCat^{\p}_{\kappa,S}(n)$ for any $\tilde{R}$-algebra $S$.

We still have a fusion product functor $\bullet\dot{\otimes}\bullet:
\OCat^\g_{\kappa,\tilde{R}}\boxtimes\OCat^\p_{\kappa,\tilde{R}}\rightarrow
\OCat^\p_{\kappa,\tilde{R}}$. Using this we can define the functors
$E,F$ and endomorphisms $X\in \End(F),T\in \End(F^2)$. This gives rise
to a quotient functor $\OCat^\p_{\kappa,\tilde{R}}(n)\twoheadrightarrow
\mathcal{H}^{\bf s}_{\kappa,\widetilde{R}}(n)\operatorname{-mod}$.
Here $\mathcal{H}^{\bf s}_{\kappa,\widetilde{R}}(n)$ is the cyclotomic
Hecke algebra over $\tilde{R}$ with parameters $q:=\exp(2\pi\sqrt{-1}(\kappa^{-1}+x_0)),
Q_i:=\exp(2\pi\sqrt{-1}(\kappa^{-1}+x_0)(s_i+x_i))$.

Using the formulas $x_0\mapsto x_0-1/e, x_i\mapsto (x_0-1/e)(s_i+x_i)-i/\ell$
we define an embedding  $\operatorname{Spec}(\tilde{R})\hookrightarrow \param$.
Let $\param_0$ be a generic plane through $\bc\in \param$.
Let $R$ be the quotient of $\tilde{R}$ that is the algebra of functions
on the preimage of $\param_0$ in $\operatorname{Spec}(\tilde{R})$.
So $R$ is isomorphic to the algebra of formal power series in
two variables.

We have two highest weight categories $\OCat^1_R:=\OCat^{\p}_{\kappa,R}(n),
\OCat^2_R:=\OCat_{\kappa,{\bf s},R}(n)$ with the same posets of simples
and with quotient functors $\pi^i_R:\OCat^i_R\twoheadrightarrow
\Cat_R:=\mathcal{H}^{\bf s}_{q,R}(n)\operatorname{-mod}$.


It is not difficult to see
that both $\pi^i_R$ become equivalences after changing the base ring from $R$
to the fraction field $\operatorname{Frac}(R)$. However, establishing
faithfulness properties of $\pi^1_R$ that are needed to apply techniques
described in \ref{SSS:deform} is difficult. The main reason why we need
a two-parametric deformation is that one can show that the functor
$\pi^1_R$ is $0$-faithful (=fully faithful on standardly filtered objects).
The following result is an extension of  \cite[Proposition 4.42]{rouqqsch}.

\begin{Prop}\label{Prop:minus1_faith}
Assume that for every point $\p\in \operatorname{Spec}(R)$ of codimension $1$,
the functor $\pi^1_{\p}:\OCat^1_{\p}\twoheadrightarrow \Cat_{\p}$ is $(-1)$-faithful,
i.e., faithful on standardly filtered objects. Then $\pi^1_R$ is $0$-faithful.
\end{Prop}

This is proved in \cite[Proposition 3.1]{VV_proof}, related but weaker results
can be found in \cite[2.4.2]{RSVV}.


\subsection{Category equivalence}
In this section, we will explain how to prove a slightly weaker (``asymptotic'') version of Conjecture \ref{Conj:VV}.
\subsubsection{Equivalence theorem}
Here we are going to provide a general result on an equivalence of two highest weight
covers of the same category.

Let $R$ be an algebra of formal power series over $\C$ and let $p$ denote the closed
point in $\operatorname{Spec}(R)$. Suppose that we have two highest weight categories
$\OCat^i_R, i=1,2,$ over $R$ such that the base changes $\OCat^i_{\operatorname{Frac}(R)}$
are split semisimple.

Let $\bar{P}^i_R\in \OCat^i_R$
be projective objects with a fixed isomorphism $\End(\bar{P}^1_R)\xrightarrow{\sim}\End(\bar{P}^2_R)$.
Let $\Cat_R$ denote the category of right $\End(\bar{P}^i_R)$-modules. So we have
quotient functors $\bar{\pi}^i_R:\OCat^i_R\rightarrow \Cat_R$. Further, let $P^i_R\subset \bar{P}^i_R, i=1,2,$
be  a direct summand with the following properties.
\begin{itemize}
\item $P^i_R$ deforms an injective object in $\OCat^i_p$.
\item We have $\bar{\pi}^1_R(P^1_R)=\bar{\pi}^2_R(P^2_R)$.
\item The base change $P^i_{\operatorname{Frac}(R)}$ generates $\OCat^i_{\operatorname{Frac}(R)}$.
\item Under the identification $\operatorname{Irr}(\OCat^1_{\operatorname{Frac}(R)})\cong
\operatorname{Irr}(\Cat_{\operatorname{Frac}(R)})\cong \operatorname{Irr}(\OCat^2_{\operatorname{Frac}(R)})$,
there is a common highest weight order on the labeling sets
$\operatorname{Irr}(\OCat^i_{\operatorname{Frac}(R)})$.
\end{itemize}

\begin{Prop}\label{Prop:hw_equiv_gen2}
Suppose that the following is true.
\begin{itemize}
\item The functor $\pi^1_R$ is $0$-faithful.
\item The functor $\pi^2_R$ is $1$-faithful.
\item For any projective $\hat{P}^2_R\in \OCat^2_R$, there is an embedding
$\hat{P}^2_R\hookrightarrow (P^2_R)^{\oplus m}$ (for some $m$)
with standardly filtered cokernel.
\end{itemize}
Then there is a highest weight equivalence $\OCat^1_R\xrightarrow{\sim}\OCat^2_R$
intertwining the quotient functors $\bar{\pi}^i_R$.
\end{Prop}

\subsubsection{Categories $\Cat^i_R$}
We will take $P^i_R:=F^n \Delta_R^{i}(\varnothing)$. The choice of the objects
$\bar{P}^i_R$ is more subtle. Let $p$ denote the closed point of $R$. It follows
from \cite[Proposition 4.42]{rouqqsch} that the functor $\bar{\pi}^2_R$ is $1$-faithful
provided the functor $\bar{\pi}^2_p$ is $0$-faithful. Take
$$\bar{P}^2_R:=P^2_R\oplus \bigoplus_\lambda P^2_R(\lambda),$$
where the summation is taken over all $\lambda\in \mathcal{P}_\ell(n)$
such that $\operatorname{codim}_\h \Supp L_p(\lambda)=1$. Using results
from \ref{SSS_KM_crystal}, \ref{SSS_Heis_cryst}, it is not difficult to describe
the set of $\lambda$ with this property explicitly. First of all, note
that $P^2_R(\lambda)$ appears as a summand of $P^2_R$ if and only if
$\Supp L^2_p(\lambda)=\h$ if and only if
$\lambda$ lies in the crystal component of $\varnothing$. Further, $\operatorname{Supp}L_p(\lambda)$
has codimension $1$ if and only if one of the following conditions holds:
\begin{itemize}
\item $\lambda$ lies in the connected component of the crystal that contains
a singular multipartition of $1$.
\item for $e=2$, $\lambda$ lies in the crystal component of $\nu$, which is
a minimal multipartition of $2$ that is not a column.
\end{itemize}
One can show that $\bar{\pi}^2_p$ is $0$-faithful, see the proof of
\cite[Proposition 8.1]{VV_proof} or \cite[Lemma 2.8]{RSVV}.

Now set $\bar{P}^1_R:=P^1_R\oplus \bigoplus_{\lambda}P^1_R(\lambda)$, where
the summation is taken over the same set of $\lambda$. What we need to prove
is an isomorphism $\End(\bar{P}^1_R)=\End(\bar{P}^2_R)$. As was shown in
\cite[Section 7]{VV_proof}, this reduces to proving the following four claims
\begin{itemize}
\item[(a1)] $\pi^1_R(P^1_R(\lambda))\cong \pi^2_R(P^2_R(\lambda))$ for any
singular $\lambda$ with $|\lambda|=1$.
\item[(b1)] $\pi^1_R(P^2_R(\nu))\cong \pi^2_R(P^2_R(\nu))$ for $\nu$ as above.
\item[(a2)] $\Hom_{\OCat^1_R(1)}(P^1_R(\lambda),P^1_R(\lambda'))\xrightarrow{\sim}
\Hom_{\mathcal{H}^{\bf s}_{q,R}(1)}(\pi^1_R P^1_R(\lambda),\pi^1_R P^1_R(\lambda'))$
for any singular $\lambda,\lambda'$ with $|\lambda|=|\lambda'|=1$.
\item[(b2)] $\End_{\OCat^1_R(2)}(P^1_R(\nu))\xrightarrow{\sim}
\Hom_{\mathcal{H}^{\bf s}_{q,R}(2)}(\pi^1_R P^1_R(\nu))$.
\end{itemize}
Recall that ``singular'' means ``annihilated by all $\tilde{e}_i$''.

The proofs of (a1),(a2) are easy. In order to prove (b1) and (b2) we need an explicit construction
of $P^1_R(\nu), P^2_R(\nu)$. It is not difficult to see that $P^2_R(\nu)$ is the projection
to the block corresponding to $\nu$ of $\operatorname{Ind}_{G(\ell,1,2)}^{\Sym_2}\Delta^2_{A,R}((2))$,
where we write $\Delta^2_{A,R}((2))$ for the Verma module over $H_R(\Sym_2)$ corresponding
to the trivial $\Sym_2$-module.  This motivates us to consider the object $\Delta^1_{A,R}((2))\dot{\otimes}\Delta^1_R(\varnothing)$ in $\OCat^1_R(2)$
and define $Q^1_R(\nu)$ as its projection to the block corresponding to $\nu$.
Clearly, (b1) holds if we replace $P^1_R(\nu)$ with $Q^1_R(\nu)$.
It turns out that $Q^1_R(\nu)\cong P^1_R(\nu)$, \cite[Proposition 7.5]{VV_proof}.

\subsubsection{Checking $(-1)$-faithfulness}
Now let us explain how to check $(-1)$-faithfulness of the functor $\pi^1_{\p}$ at codimension
$1$ points of $\operatorname{Spec}(R)$. This is done using categorical actions and their crystals.
This approach  only allows to establish a weaker version of Conjecture \ref{Conj:VV}: we need
to assume that $s_i\gg n$ for all $i$.

Namely, we still have functors $\underline{E},F$ on $\OCat^1_{\p}(\leqslant N)$ (for some $N\gg n$)
and the natural  transformations $X\in \End(F),T\in \End(F^2)$. These data give rise to a
(restricted) categorical Kac-Moody action on $\OCat^1_{\p}(\leqslant N)$. The algebra
acting is $\g_{\tilde{\bc}}$, where $\tilde{\bc}\in \param\otimes \mathbf{k}_{\p}$ is the element
corresponding to the homomorphism $\C[\param]\rightarrow R\rightarrow \mathbf{k}_{\p}$.
Equivalently, if $\p$ is an intersection of a hyperplane in $\param$ with $\operatorname{Spec}(R)$,
then we can take  a Weil generic point $\tilde{\bc}$ in this hyperplane.
So we have three options.
\begin{itemize}
\item[(i)] $\g_{\tilde{\bc}}=\mathfrak{sl}_\infty^{\oplus \ell}$. In this case,
$\pi^1_{\p}$ is an equivalence.
\item[(ii)] $\p$ is the intersection of the hyperplane $y_0=-\frac{1}{e}$
with $\operatorname{Spec}(R)$. In this case $\g_{\tilde{\bc}}=\hat{\mathfrak{sl}}_e^{\oplus \ell}$.
\item[(iii)] $\p$ is the intersection of $\operatorname{Spec}(R)$ with the hyperplane of the form $s_i-s_j=m$ (where $m$
is an integer with $|m|<N$ and $i,j$ are different elements of $\{0,1,\ldots,\ell-1\}$).
In this case $\g_{\tilde{\bc}}=\hat{\mathfrak{sl}}_\infty^{\oplus \ell-1}$.
\end{itemize}

In (i), there is nothing to prove. To deal with (ii) and (iii) we use the following
result.

\begin{Prop}\label{Prop:-1_faith}
Let $\OCat(\leqslant N)$ be a restricted highest
weight $\g_{\tilde{\bc}}$-categorification (over an arbitrary field
$\mathbf{k}$) of the level $\ell$
Fock space $\mathcal{F}_{\tilde{\bc}}$ as in (ii) or (iii).
Pick $n\ll N$. Then the quotient functor $\pi$ defined by
the projective $F^n{\bf 1}$ is $(-1)$-faithful.
\end{Prop}
\begin{proof}
We are going to sketch the proof (a complete proof can be found in
\cite[Sections 6,9.2]{VV_proof}). The proof is in several steps,
all but the last one work for all $\bc$.

{\it Step 1}. Let $\mathcal{P}_\ell^0$ denote the connected component
of $\varnothing$ in the $\g_\bc$-crystal. What we need to show is that
if $\lambda\not\in\mathcal{P}_\ell^0$, then $\Hom_{\OCat}(L(\lambda),T(\mu))=0$.
It is easy to reduce to the case when $\lambda$ is singular. Note that in
this case $\Hom_{\OCat}(L(\lambda),T(\mu))=0$ provided $T(\mu)$
is a direct summand of a tilting object of the form $F_i T(\mu')$.
Using results from \cite{hw_str}, we can determine one summand
of $F_i T(\mu')$ combinatorially starting from $\mu'$. For this, we need
a so  called {\it dual} crystal structure on $\mathcal{P}_\ell$.
It is defined as follows. We write down the $i$-signature of $\mu'$ but now cancel
$+-$ rather than $-+$. In this way we get crystal operators $\tilde{e}_i^*,
\tilde{f}_i^*$. By \cite[Lemma 4.2]{VV_proof},
$T(\tilde{f}_i^* \mu')$ is a direct summand of $F_i T(\mu')$.
We conclude that, for singular $\lambda$, we have $\Hom_{\OCat}(L(\lambda),T(\mu))=0$
unless $\tilde{e}_i^* \mu=0$ for all $i$. In this case, we say that
$\mu$ is {\it cosingular}. So we need to show that $\Hom_{\OCat}(L(\lambda),T(\mu))=0$
when $\lambda\neq \varnothing$ is singular, and $\mu$ is cosingular.

{\it Step 2}. For $\lambda$ with $\tilde{e}_i\lambda=0$ we can define
its reflection $\sigma_i\lambda$ by $\sigma_i\lambda:=\tilde{f}_i^k \lambda$,
where $k$ is minimal with $\tilde{f}_i^{k+1}\lambda=0$. Similarly,
for $\mu$ with $\tilde{e}_i^*\mu=0$, we can define $\tilde{\sigma}_i^*\mu$.
When $\lambda$ is singular, and $w$ is in the Weyl group of $\g_\bc$
(such that $|w\lambda|\leqslant N$), then $w\lambda$ is well-defined,
in particular, is independent of the choice of a reduced expression for
$w$. Similarly, $w^*\mu$ is well-defined for a cosingular $\mu$.
A key observation for us is that, if $\tilde{e}_i\lambda=0,\tilde{e}_i^*\mu=0$,
then $\dim \Hom_{\OCat}(L(\lambda),T(\mu))=\dim \Hom_{\OCat}(L(\sigma_i\lambda),
T(\sigma_i^*\mu))$. It follows that for a singular $\lambda$ and a cosingular
$\mu$ we have
\begin{equation}\label{eq:dim_Hom}
\dim \Hom_{\OCat}(L(\lambda),T(\mu))=\dim \Hom_{\OCat}(L(w\lambda), T(w^*\mu)).
\end{equation}

{\it Step 3}. If we are in situation (iii), then actually $\lambda\not\leqslant \mu$
for any singular $\lambda$ and any cosingular $\mu$, this can be deduced from
Examples \ref{Ex:level2_cryst_1},\ref{Ex:level2_cryst_2}. So we are done by Step 1.
In case (ii), there is $w$ such that $w\lambda\not\leqslant w\mu$. So we are done
by (\ref{eq:dim_Hom}). It is here that we use that $N$ is large enough.
\end{proof}

\subsection{Multiplicities}
Conjecture \ref{Conj:VV} allows to compute the multiplicities
in $\OCat_{\kappa,{\bf s}}(n)$ because the multiplicities in
$\OCat_\kappa^{\p}$ are known, they are given by (the values
at $q=1$ of) suitable parabolic Kazhdan-Lusztig polynomials.

There is an alternative way to present this result. Namely, in
\cite{Uglov}, Uglov defined a  $U_v(\hat{\slf}_e)$-action
on $\mathcal{F}_{\kappa,\bf s}(v)$, where $U_v(\hat{\slf}_e)$ is the
quantum group over $\C(v)$. He introduced
a $\C[v^{-1}]$-lattice in   $\mathcal{F}_{\kappa,\bf s}(v)$
and a bar-involution $\mathcal{F}_{\kappa,\bf s}(v)$. Using these data,
one can define a so called dual canonical basis,
$b_{\bf s}(\lambda)\in \mathcal{F}_{\kappa,\bf s}[v^{-1}]$
indexed by $\lambda\in \mathcal{P}_\ell$. As was checked
in \cite[Section A4]{VV}, the specialization of $b_{\bf s}(\lambda)$
to $v=1$ coincides $[L_{\kappa,{\bf s}}(\lambda)]\in K_0^{\C}(\OCat_{\kappa,{\bf s}})$.
This was conjectured by Rouquier in \cite[Section 6]{rouqqsch}.

\subsection{Complements}
Above in this section we were dealing with the case when $\kappa=-\frac{1}{e}$,
$s_0,\ldots,s_{\ell-1}\in \Z$. Let us explain how to describe the category
$\mathcal{O}_{\kappa,{\bf s}}(n)$ in general. We can assume that
there is a single equivalence class in $\{0,1,\ldots,\ell-1\}$ with respect
to $\sim_{\kappa,{\bf s}}$. We can always reduce to this case
using Proposition \ref{Prop:cat_O_decomp}.

Let us consider the case when $\kappa$ is irrational. In this case
the category $\mathcal{O}_{\kappa,{\bf s}}(n)$ is equivalent
to the sum of suitable blocks in a suitable parabolic
category $\mathcal{O}$ for $\mathfrak{gl}_m$ (for some $m$).
This was checked in \cite[Section 6]{GL} under  the faithfulness restrictions
on the parameters $s_0,\ldots,s_{\ell-1}$. These restrictions
can be removed by considering extended quotients, similarly
to \ref{SSS_ext_quot}.

Now consider the case when $\kappa$ is rational. Using the equivalence
relating positive and negative $\kappa$, we can assume that $\kappa=-\frac{r}{e}$,
where $r>0$ is coprime to $e$. The condition that $\{0,1,\ldots,\ell-1\}$
is a single conjugacy class means that $r s_i\in \Z$ (up to a common shift
of the $s_i$'s, as usual). We are going to prove the following result.

\begin{Prop}\label{Prop:orders}
The category $\mathcal{O}_{\kappa,{\bf s}}(n)$ is equivalent to the category
$\mathcal{O}_{-1/e,{\bf s}'}(n)$ for a suitable collection ${\bf s}'\in \Z^\ell$.
The equivalence sends $\Delta_{\kappa,{\bf s}}(\lambda)$ to
$\Delta_{-1/e, {\bf s}'}(\sigma \lambda)$, where $\sigma\in \Sym_\ell$
and $\sigma\lambda$ is the $\ell$-partition obtained from $\lambda$
by permuting $\lambda^{(0)},\ldots,\lambda^{(\ell-1)}$ according to $\sigma$.
\end{Prop}

Since the proof has not appeared elsewhere, we are going to provide it here
in full.

We start with a lemma that was proved in \cite{cher_abloc} (the proof is an elementary combinatorics).

\begin{Lem}\label{Lem:param_perturb}
Let $\bc=(\kappa, {\bf s})$ be a parameter with $\kappa=-\frac{r}{e}, {\bf s}\in (\frac{1}{r}\Z)^\ell$.
Then, for $M\gg 0$ and integers $0<m_0<\ldots<m_{\ell-1}\ll M$,
the orders on the boxes coincide and hence the orders $\preceq_\bc, \preceq_{\bc'}$ are equivalent, where $\bc':=(M\kappa, s_0+m_0/\kappa, s_1+m_1/\kappa,\ldots, s_{\ell-1}+m_{\ell-1}/\kappa)$.
\end{Lem}

Next, we have the following result.

\begin{Lem}\label{Lem:equiv_known} Let $\bc$ be a parameter given by $(\kappa,{\bf s})$.
Suppose that $\bc$ is spherical. Then, for any $\sigma\in \mathfrak{S}_\ell$,
there is an equivalence $\mathcal{O}_\bc\xrightarrow{\sim}\mathcal{O}_{\sigma.\bc}$ that maps
$\Delta_\bc(\lambda)$ to $\Delta_{\sigma.\bc}(\sigma\lambda)$, and preserves the supports.
\end{Lem}
This  lemma was proved in  \cite[4.1.4]{GL}. Here $\sigma.\bc$ denotes the parameter
obtained from $\bc$ by permuting the $s_i$'s according to $\sigma$.

\begin{proof}[Proof of Proposition \ref{Prop:orders}]
We will produce $\bc'=(-1/e,{\bf s}')$ in several steps.

Lemma \ref{Lem:param_perturb} and (1) of Lemma \ref{Lem:equiv_known} allow us to pass from a parameter $\bc$
to a Zariski generic (and, in particular, spherical) parameter $\bc^1$  (since the spherical parameters are known explicitly, \cite{DG}, one can write $\bc^1$ explicitly as well). Fix an integer $r$ with $0\leqslant r<e$. In each of the $\ell$ diagrams we have precisely one diagonal of boxes with shifted content ($s^1_i+x-y$, for the box $(x,y,i)$) congruent $r$ modulo $e$ and lying in $[r,e+r-1)$. Let $(i_0,\ldots,i_{\ell-1})$ be the indexes of the diagrams such that the diagonals are in the increasing  order (with respect to the preorder on boxes from \ref{SSS_O_order}),
note that the ordering of indexes is independent of the choice of $r$. Let $\sigma\in \mathfrak{S}_\ell$ be defined  by
$\sigma^{-1}(j)=i_j$. So the analogously defined sequence
of indexes for the parameter $\bc^2:=\sigma.\bc^1$ is $(0,1,\ldots,\ell-1)$. We get an equivalence $\OCat_{\bc^1}\xrightarrow{\sim}\OCat_{\bc^2}$
by using (2) of Lemma \ref{Lem:equiv_known}. Set $\bc^3:=(\kappa^2,s^3_0,\ldots,s^3_{\ell-1})$,
where ${\bf s}^3$ is such that all differences $s^3_i-s^3_j$ are integral and the order on the diagonals is the same
as before, it is clear that such ${\bf s}^3$ exists.
So $\bc^2,\bc^3$ are equivalent and we get a highest weight equivalence $\OCat_{\bc^2}\xrightarrow{\sim} \OCat_{\bc^3}$ by (1) of Lemma \ref{Lem:equiv_known}. Now set $\bc':=(-\frac{1}{e}, s_0^3,\ldots,s^3_{\ell-1})$.
\end{proof}



\begin{thebibliography}{99}
\bibitem[A]{Ariki} S.~Ariki, {\it Representations of quantum algebras and combinatorics of
Young tableaux}. University Lecture Series, 26, AMS.
\bibitem[AK]{AK} S.~Ariki and K.~Koike, {\it A Hecke algebra of $(\Z/r\Z)\wr \mathfrak{S}_n$ and construction of its irreducible representations,} Adv.Math. {\bf 106} (1994), 216--243.
\bibitem[BK]{BK} A. Berenstein, D. Kazhdan. {\it Geometric and unipotent crystals. II. From unipotent bicrystals to crystal bases}.  Quantum groups,  13-88, Contemp. Math., 433, Amer. Math. Soc., Providence, RI, 2007.
\bibitem[BEG1]{BEG1} Yu. Berest, P. Etingof, V. Ginzburg. {\it Cherednik algebras
and differential operators on quasi-invariants}.  Duke Math. J. 118(2003),  279-337.
\bibitem[BEG2]{BEG2} Yu. Berest, P. Etingof, V. Ginzburg. {\it Finite dimensional representations
of rational Cherednik algebras}.  Int. Math. Res. Not. 19(2003),  1053-1088.
\bibitem[BC]{BC} Y.~Berest and O.~Chalykh, {\it Quasi-invariants of complex reflection groups}, Comp. Math. {\bf 147} (2011), 965-1002.
\bibitem[BE]{BE} R. Bezrukavnikov, P. Etingof, {\it Parabolic induction and restriction
functors for rational Cherednik algebras}.  Selecta Math.,  14(2009), 397-425.
\bibitem[BMR]{BMR} M.~Brou\'{e}, G.~Malle and R.Rouquier, {\it Complex reflection groups, braid groups, Hecke algebras},  J. Reine Angew. Math. {\bf 500}  (1998), 127-190.
\bibitem[CR]{CR} J. Chuang and R. Rouquier, {\it Derived equivalences for symmetric groups and
$\mathfrak{sl}_2$-categorifications}. Ann. Math. (2)  167(2008), n.1, 245-298.
\bibitem[DG]{DG} C. Dunkl, S. Griffeth. {\it Generalized Jack polynomials and the
representation theory of rational Cherednik algebras}. Selecta Math. 16(2010), 791-818.
\bibitem[DJ]{DJ} R. Dipper, G. James. {\it The q-Schur algebra}. Proc. London Math. Soc. (3)  59  (1989),  no. 1, 23-50.
\bibitem[EG]{EG} P.~Etingof and V.~Ginzburg. {\it Symplectic reflection algebras, Calogero-Moser space,
and deformed Harish-Chandra homomorphism}, Invent. Math. {\bf 147} (2002), 243-348.
\bibitem[EGL]{EGL} P. Etingof, E. Gorsky, I. Losev, {\it Representations of  Cherednik algebras with minimal
support and torus knots}. arXiv:1304.3412. Adv. Math. 227 (2015), 124-180.
\bibitem[GGOR]{GGOR} V. Ginzburg, N. Guay, E. Opdam and R. Rouquier, {\it On the category $\mathcal{O}$ for rational
Cherednik algebras}, Invent. Math., {\bf 154} (2003), 617-651.
\bibitem[GL]{GL} I. Gordon, I. Losev, {\it On category $\mathcal{O}$ for cyclotomic rational
Cherednik algebras}.  J. Eur. Math. Soc. 16 (2014), 1017-1079.
\bibitem[Lo1]{fun_iso} I. Losev, {\it On isomorphisms of certain functors for  Cherednik algebras}. Repres. Theory, 17 (2013), 247-262.
\bibitem[Lo2]{cryst} I. Losev, {\it Highest weight $\mathfrak{sl}_2$-categorifications I: crystals}. Math. Z. 274(2013), 1231-1247.
\bibitem[Lo3]{hw_str} I. Losev, I. Losev, {\it Highest weight $\mathfrak{sl}_2$-categorifications II: structure theory}.
arXiv:1203.5545. To appear in Trans. Amer. Math. Soc.
\bibitem[Lo4]{VV_proof} I. Losev, {\it Proof of Varagnolo-Vasserot conjecture on cyclotomic categories $\mathcal{O}$}. arXiv:1305.4894. Accepted by Selecta Math.
\bibitem[Lo5]{cher_abloc} {\it Abelian localization for cyclotomic Cherednik algebras}. arXiv:1402.0224.
To appear in IMRN.
\bibitem[Lo6]{rouq_der} I. Losev. {\it Derived equivalences for Rational Cherednik algebras}. arXiv:1406.7502.
\bibitem[Lo7]{Hecke_fin_dim} I. Losev. {\it Finite dimensional quotients of Hecke algebras}. arXiv:1407.6375.
Algebra and Number theory, 9(2015), 493-502.
\bibitem[Lo8]{Cher_supp} I. Losev. {\it Supports of simple modules in cyclotomic Cherednik categories O }.
arXiv:1509.00526.
\bibitem[Lu]{Lusztig} G. Lusztig.  {\it Introduction to quantum groups}, Progress in Mathematics, 110,
Birkh\"{a}user Boston, Boston, 1993.
\bibitem[M]{Martin} S. Martin. {\it Schur algebras and representation theory}.
Cambridge Tracts in Mathematics, 112. Cambridge University Press, Cambridge, 1993. xvi+232 pp.
\bibitem[R1]{rouquier_ICM} R. Rouquier, {\it Derived equivalences and finite dimensional algebras}. Proceedings of
ICM 2006.
\bibitem[R2]{rouqqsch} R. Rouquier, {\it $q$-Schur algebras for complex reflection groups}. Mosc. Math. J. 8 (2008), 119-158.
\bibitem[R3]{Rouquier_2Kac} R. Rouquier, {\it 2-Kac-Moody algebras}. arXiv:0812.5023.
\bibitem[RSVV]{RSVV} R. Rouquier, P. Shan, M. Varagnolo, E. Vasserot. {\it Categorification and cyclotomic rational double affine Hecke algebras}. arXiv:1305.4456.
\bibitem[Sh]{Shan} P. Shan. {\it Crystals of Fock spaces and cyclotomic rational double affine Hecke algebras}.
Ann. Sci. Ecole Norm. Sup. {\bf 44} (2011), 147-182.
\bibitem[SV]{SV} P.~Shan and E.~Vasserot. {\it Heisenberg algebras and rational double affine Hecke algebras}. J. Amer. Math. Soc. 25(2012), 959-1031.
\bibitem[ST]{ST}  G. C. Shephard and J. A. Todd,
{\it Finite unitary reﬂection groups}, Canad. J. Math. 6 (1954), 274-304.
\bibitem[U]{Uglov} D. Uglov. {\it Canonical bases of higher-level q-deformed Fock spaces and Kazhdan-Lusztig polynomials},
in ``Physical combinatorics (Kyoto, 1999)'', pp. 249–299, Birkhauser, 2000.
\bibitem[VV]{VV} M.~Varagnolo and E.~Vasserot, {\it Cyclotomic double affine Hecke algebras and affine parabolic category $\mathcal{O}$}, Adv. Math. {\bf 225} (2010), 1523--1588.
\bibitem[W]{Webster} B. Webster. {\it Rouquier's conjecture and diagrammatic algebra}. arXiv:1306.0074.
\end{thebibliography}
\end{document}